\documentclass[11pt]{article}
\newfont{\bb}{msbm10}

\usepackage{amsfonts}
\usepackage{amsmath,amssymb,amsthm}
\usepackage{booktabs}
\usepackage{amsbsy}
\usepackage{graphicx,color}

\newtheorem{remark}{Remark}
\newtheorem{theorem}{Theorem}

\newtheorem{lemma}{Lemma}[section]
\def\Div{{\rm div}\,}
\newcommand{\bu} {\mathbf{u}}
\newcommand{\eps} {\varepsilon}

\newcommand{\la} {\langle}
\newcommand{\ra} {\rangle}
\newcommand{\bn} {\mathbf{n}}

\newcommand{\bv} {\mathbf{v}}

\newcommand{\bw} {\mathbf{w}}
\newcommand{\bF} {\mathbf{F}}

\newcommand{\bx} {\mathbf{x}}
\newcommand{\be} {\mathbf{e}}
\newcommand{\blf} {\mathbf{f}}
\newcommand{\bg} {\mathbf{g}}
\newcommand{\QQ} {\mathbb{Q}}
\newcommand{\VV} {\mathbb{V}}

\newcommand{\XX} {\mathbb{X}}

\newcommand{\bI} {\mathbf{I}}

\newcommand{\bD} {\mathbf{D}}
\newcommand{\dO} {{\partial\Omega}}

\newcommand{\bA} {{\bf A}}
\newcommand{\bB} {{\bf B}}

\newcommand{\bsigma}{\mbox{\boldmath$\sigma$\unboldmath}}
\newcommand{\bpsi}{\mbox{\boldmath$\psi$\unboldmath}}
\newcommand{\bphi}{\mbox{\boldmath$\phi$\unboldmath}}
\newcommand{\bxi}{\mbox{\boldmath$\xi$\unboldmath}}

\newcommand{\rev}[1]{{\color{black}#1}}

\begin{document}

\thispagestyle{empty}
\date{}

\title{A quasi-Lagrangian finite element method for the Navier-Stokes equations in a time-dependent domain\thanks{This work has been supported by the Russian Science Foundation (RSF) grant {14-31-00024}.}}

\author{
Alexander Lozovskiy\thanks{Institute of Numerical Mathematics RAS; {\tt saiya-jin@yandex.ru}}
\and
Maxim A. Olshanskii\thanks{Department of Mathematics, University of Houston; {\tt molshan@math.uh.edu}}
\and
Yuri V. Vassilevski\thanks{Institute of Numerical Mathematics RAS, Moscow Institute of Physics and Technology,  Sechenov University; {\tt yuri.vassilevski@gmail.com}} }

\maketitle

\markboth{}{FEM for NSE in a time-dependent domain}

\begin{abstract}
The paper develops a finite element method for the Navier-Stokes equations of incompressible viscous fluid in a time-dependent domain. The method builds on a quasi-Lagrangian formulation of the problem. The paper provides stability and convergence analysis of the fully discrete (finite-difference in time and finite-element in space) method. The analysis does not assume any CFL  time-step restriction, it rather needs mild conditions of the form $\Delta t\le C$, where $C$ depends only on problem data, and
$h^{2m_u+2}\le c\,\Delta t$,  $m_u$ is polynomial degree of velocity finite element space. Both conditions result from a numerical treatment of practically important non-homogeneous boundary conditions. The theoretically predicted convergence rate is confirmed by a set of numerical experiments.  Further we apply the method to simulate a flow in a simplified model of the left ventricle of a human heart, where the ventricle wall dynamics is reconstructed from a sequence of contrast enhanced Computed Tomography images.
\end{abstract}

\section{Introduction}\label{intro}

Fluid flows in time-dependent domains are ubiquitous in nature and engineering.  In many cases, finding the domain evolution is part of the problem and the mathematical model couples fluid and structure dynamics.
Examples include fluid--structure interaction problems for blood flow in compliant vessels, flows around  turbine blades or fish locomotion. In other situations, one may assume that the motion of the domain is given and one has to recover the induced fluid flow.
One example of a problem, which often assumes \textit{a priori} information about the flow domains evolution, is the blood flow simulation in a human heart when the (patient-specific) motion of the heart walls is recovered from a sequence of medical images
 \cite{saber2003progress,long2008subject,schenkel2009mri,mihalef2011patient,doenst2009fluid,chnafa2014image,mittal2016computational,su2016cardiac}.
Nowadays  numerical simulations are commonly used to understand fluid dynamics and predict statistics of practical interest in this and other applications. 
In the present paper, we develop  a finite element~{(FE)} method  for a quasi-Lagrangian formulation
of the incompressible Navier-Stokes equations in a moving domain.
We consider an implicit--explicit method, i.e an implicit method with  advection field in the inertia term lagged in time. For the spatial discretization we employ inf-sup stable pressure--velocity elements.

Several techniques have been introduced in the literature to overcome  numerical difficulties due to the evolution of the domain. This includes space--time finite element formulations, immersed boundary methods,
level-set method,  fictitious domain method, unfitted finite elements, and arbitrary Lagrangian--Eulerian  (ALE) formulation, see, e.g., \cite{masud1997space,tezduyar1992new,osher2006level,glowinski1999distributed,peskin1977numerical,gross2007extended,hirt1974arbitrary,nobile1999stability,duarte2004arbitrary}. In this paper we analyze a finite element method based on a quasi-Lagrangian formulation of the equations in the reference domain.  Related analysis of finite element methods for parabolic or  fluid equations in moving domains can be found in several places in the literature. We note that  well-posedness of  space-time weak saddle-point formulations of the (Navier--)Stokes equations is a subtle question, see  the recent treatment in \cite{guberovic2014space} for the case of a steady domain.  A rigorous  stability and convergence analysis of space--time (FE) methods for fluid problems seems to be largely lacking.  Scalar problems have been understood much better; for example, a space--time discontinuous FE method for advection--diffusion problems on time-dependent domains was {analyzed} in \cite{sudirham2006space}. ALE and Lagrangian finite element methods are more  amenable to analysis.
The  stability of ALE finite element methods for parabolic evolution problems  was treated  in \cite{nobile1999stability}.
 The authors of \cite{MartinEtAl2009}  {analyzed} the convergence of a finite element ALE method for the Stokes equations in a time-dependent domain when the motion of the domain is given. The analysis~\cite{MartinEtAl2009} imposes time step restriction,  assumes zero velocity boundary condition and certain smoothness assumptions for the finite element displacement field.
A closely related method to the one studied here was considered in \cite{DLOVpaper}. However, that paper introduced an assumption that a divergence free extension of a boundary
 condition function to the computational domain is given. This assumption is not always practical and the present paper  avoids it. Moreover, this paper develops error analysis, while the thrust of  \cite{DLOVpaper} was the stability analysis and the numerical recovery procedure of the domain motion from medical images.

In the present paper, we analyze a quasi-Lagrangian FE formulation that is closely related to an ALE formulation, although
they are not equivalent.  In the present approach we discretize equations in a reference time-independent domain. The geometry evolution is accounted in time-dependent coefficients. Inertia terms are further linearized so that only a  system of linear algebraic equations is solved on each time step. We consider practically relevant  boundary conditions, which result in non-homogeneous velocity on the boundary.   For this method we prove  numerical stability and optimal order error estimates in the energy norm without a CFL condition on the time step. Divergence-free condition enforced in the reference domain leads to time dependent functional spaces; this and handling non-homogeneous boundary conditions are two main difficulties that we overcome in the analysis. For the numerical stability bound we shall need the condition on the space mesh size and time step of the form $h^{2m_u+2}\le c\,\Delta t$, where  $m_u\ge1$ is polynomial degree of velocity finite element space and $c$ is a constant. We note that if one assumes zero boundary conditions for velocity, which is a standard assumption for FE stability bounds in steady domains, then the results of the paper hold without the above  conditions on $h$ and $\Delta t$. In our opinion, homogeneous boundary conditions is not a suitable assumption for the FE analysis in evolving domains, see discussion in section~\ref{sec_FE}.

Thus the paper advances the known analysis by including inertia effect, removing CFL time-step restriction, handling physically meaningful  boundary conditions, and making no  further auxiliary assumptions except the following one: the domain evolution is given  \textit{a priori} by a smooth mapping from a reference domain to a physical domain  and exact quadrature rules are applied in the reference domain, i.e. we do not analyse possible errors due to  inexact numerical integration. The mapping is \emph{not} necessarily Lagrangian in the internal points, but it has to be Lagrangian for those parts of the boundary, where correct tangential velocity boundary values are important. Theoretical results are illustrated numerically  for an example with a synthetic known solution.  We further illustrate the performance of the numerical method by applying it to blood flow simulation in a simplified model of the human left ventricle. The domain motion in this example is reconstructed from a sequence of ceCT images of a real patient heart over one cardiac cycle. The reconstruction procedure is described in detail in our preceding paper~\cite{DLOVpaper}.

The remainder of this paper is organized as follows. In section~\ref{s_model} we review the mathematical model, including  governing equations and boundary conditions, and some useful results for this model found in the literature. We recall the energy balance satisfied by smooth solutions. A suitable weak formulation is introduced. Based on the weak formulation, in section~\ref{sec_FE} we introduce the finite element method.
Non-homogeneous boundary conditions are interpolated numerically. Energy stability estimate for the finite element method
is shown in section~\ref{s_stab}. Optimal order error bound for the method is demonstrated in section~\ref{sec_anal}.
Section~\ref{s_num} collects results of numerical experiments. Some closing remarks can be found in the summary and outlook section~\ref{s_outlook}.

\section{Mathematical model}\label{s_model}
Consider a time-dependent domain $\Omega(t)\subset\mathbb{R}^d$, $d=2,3$, occupied by  fluid.
To formulate a flow problem, we introduce the reference domain $\Omega_0=\Omega(0)$ and a   mapping from the space--time cylinder $Q:=\Omega_0\times[0,T]$ to the physical domain,
\[
\bxi~:~Q\to Q^{\rm phys}:=\bigcup_{t\in[0,T]}\Omega(t)\times\{t\}.
\]
The mapping is assumed to be level-preserving, i.e. $\bxi(\Omega_0\times\{t\})=\Omega(t)$ for all $t\in[0,T]$.
We assume also that the evolution of $\Omega(t)$ is sufficiently smooth such that $\bxi\in C^3(Q)^d$.
{Denote} the spatial gradient matrix of $\bxi$ by $\bF=\nabla_{\bx}\bxi$, and $J:=\mbox{det}(\bF)$.
Furthermore, we assume that  there exist such positive reals $C_F,c_J$ that
\begin{equation}\label{assumption}
\inf_{Q}J\ge c_J>0,\quad\sup_{Q}(\|\bF\|_F+\|\bF^{-1}\|_F)\le C_F,\quad \text{with}~\|\bF\|_F:=\mbox{tr}(\bF\bF^T)^{\frac12}.
\end{equation}


The dynamics of incompressible Newtonian fluid can be described in terms of the velocity vector field $\hat{\bu}(\bx,t)$ and the pressure function $\hat{p}(\bx,t)$ defined in $\Omega(t)$ for $t\in[0,T]$.
This paper studies a finite element method for fluid equations formulated in the reference domain.
For $\bu=\hat{\bu}\circ\bxi$, $p=\hat{p}\circ\bxi$ defined in $Q$,
the fluid dynamics is given by the following set of equations:
\begin{equation}\label{NSE}
\left\{
\begin{aligned}
\bu_t-
J^{-1}\Div(J(\widehat\bsigma\circ\bxi)\bF^{-T})+(\nabla \bu) (\bF^{-1}(\bu-\bxi_t))
 &=\blf\\
 \Div(J\bF^{-1}\bu)&= 0
 \end{aligned}\right. \quad\text{in}~~Q,
 \end{equation}
with  body forces $\blf=\hat{\blf}\circ\bxi$ and the initial condition
$\bu ({\bf x}, 0) = \bu _0({\bf x}) \; \text{ in } \Omega_0$.
We assume the fluid to be Newtonian, with the kinematic viscosity parameter $\nu$. The constitutive
relation  in the reference domain reads
\begin{equation}\label{constit_f}
\widehat\bsigma\circ\bxi=-p\bI+\nu(\nabla\bu \bF^{-1}+\bF^{-T}(\nabla\bu)^T)~~~\text{in}~Q.
\end{equation}

\subsection{Boundary conditions}
We distinguish between the no-slip  $\dO^{ns}(t)$, Dirichlet $\dO^{D}(t)$ and outflow $\dO^{N}(t)$ parts of the boundary, and $ \dO(t)=\dO^{ns}(t)\cup\dO^{D}(t)\cup\dO^{N}(t)$.
On $\dO^{ns}(t)$ we impose no-penetration no-slip boundary condition, i.e. the fluid velocity on $\dO(t)$ is equal to the material velocity of the boundary (see the discussion below),
\begin{equation}\label{cont_int}
\hat{\bu}=\bxi_t\circ\bxi^{-1}\quad\mbox{on}~\dO^{ns}(t),
\end{equation}
while {on} $\dO^{D}(t)$ and $\dO^{N}(t)$ we prescribe Dirichlet and Neumann conditions,
\begin{equation}\label{bc_DN}
\hat{\bu}=\hat\bu_D\quad\mbox{on}~\dO^{D}(t),\quad \hat\bsigma\hat\bn = \hat{\bg} \quad\mbox{on}~\dO^{N}(t).
\end{equation}
Here $\hat\bu_D$ is a given velocity  and $\hat\bn$ is the exterior unit normal vector on {$\dO(t)$}.  If  $\dO^{N}(t)=\emptyset$ for some $t\in[0,T]$ we assume $\int_{\dO^{ns}(t)}\hat\bn\cdot\bxi_t\circ\bxi^{-1}\,\mathrm{d} s+\int_{\dO^{D}(t)}\hat\bn\cdot\hat\bu^D\,\mathrm{d} s=0$.

In the reference domain, define $\dO^{D}_0=\xi^{-1}\dO^{D}(t)$, $\dO^{N}_0=\xi^{-1}\dO^{N}(t)$, $\dO^{ns}_0=\xi^{-1}\dO^{ns}(t)$. We assume that $\dO^{D}_0$, $\dO^{D}_0$, $\dO^{D}_0$ are independent of $t$.

\begin{remark}\rm
The normal velocity of the boundary $\dO(t)$ is $v_\Gamma=\hat\bn\cdot(\bxi_t\circ\bxi^{-1})$. However, the material tangential velocity of the boundary is defined by the tangential part of $\bxi_t$ only if $\bxi$ is the Lagrangian
mapping, i.e. $\bxi(\bx,t)$, $t\in[0,T]$, defines the material trajectory for $\bx\in\Omega_0$ (or at least for $\bx\in\partial\Omega_0$).
In some applications such Lagrangian mapping is not available, and in this case \eqref{cont_int} may produce spurious tangential velocities on the boundary. For example, this may happen if $\bxi$ is reconstructed from medical images.  Thus, in practice one may or may not amend \eqref{cont_int} based on any additional information
about the tangential motions for a better model.
\end{remark}

\subsection{An extension result}\label{s_ext}
The solvability of the problem \eqref{NSE}--\eqref{constit_f}  and the existence of its weak solutions is treated, for example, in~\cite{miyakawa1982existence}.
Moreover, it is shown in ~\cite{miyakawa1982existence} that for smoothly evolving $\Omega(t)$ the mapping $\bxi$ can be chosen in such a way that $J$ depends only on $t$. From numerical viewpoint, such a mapping $\bxi$ may not be practically available, and so we allow $J$ to vary in time and space. However, the following corollary of this result is important for us, see Theorem 4.4. in~ \cite{miyakawa1982existence}:
Assume $|\Omega(t)|=|\Omega(0)|$, then there exists $\hat\bv_1\in C^2(\overline{Q^{\rm phys}})^d$ such that $\hat\bv_1=\bxi_t\circ\bxi^{-1}$ on $\dO(t)$ and $\Div\hat\bv_1=0$ in $\Omega(t)$ for $t\in[0,T]$.
The condition $|\Omega(t)|=|\Omega(0)|$ is satisfied in the case of $\dO(t)=\dO^{ns}(t)$ for all $t\in[0,T]$. Indeed, the Reynolds transport theorem and  the incompressibility assumption for fluid  imply
\begin{equation*}\label{Omega(t)}
\frac{d}{dt}|\Omega(t)|=\frac{d}{dt}\int_{\Omega(t)}\,\mathrm{d}\bx
=\int_{\dO(t)}v_\Gamma\,\mathrm{d}s=
\int_{\dO(t)}\hat{\bn}\cdot\hat{\bu}\,\mathrm{d}s=\int_{\Omega(t)}\Div\hat{\bu}\,\mathrm{d}\bx=0.
\end{equation*}
For the finite element analysis in this paper we \textit{assume} that $\hat\bv_1$ can be taken $C^3$-smooth.
We define smooth function $\bv_1=\hat\bv_1\circ\bxi$ that satisfies
\begin{equation}\label{aux_fun}
\bv_1\in C^3(Q)^d,\quad \Div(J\bF^{-1}\bv_1)= 0~\text{in}~\Omega_0,\quad \bv_1=\bxi_t~\text{on}~\dO_0.
\end{equation}

We stress that we need the result about existence of $\bv_1$ for the finite element \emph{analysis}, but one never needs to know or compute $\bv_1$ for the implementation of the FE method.

\subsection{Energy equality}~\label{s_energy} In this section, we assume no-penetration no-slip boundary condition
\eqref{cont_int}  imposed on the whole boundary, i.e. $\dO(t)=\dO^{ns}(t)$.
By $(\cdot,\cdot)$ we denote the $L^2(\Omega_0)$ scalar product, and  $\|\cdot\|$ denotes the $L^2(\Omega_0)$ norm.
For vector fields $\bv,\bu:\Omega_0\to\mathbb{R}^d$ and tensor fields $\bA,\bB:\Omega_0\to\mathbb{R}^{d\times d}$,
we use the same notation to denote  $(\bu,\bv)=\int_{\Omega_0}\bu^T\bv\,\mathrm{d}\bx$ and $(\bA,\bB)=\int_{\Omega_0}\text{tr}(\bA\bB^T)\,\mathrm{d}\bx$,
and obviously $\|\bu\|:=(\bu,\bu)^\frac12$, $\|\bA\|:=(\bA,\bA)^\frac12$.
We shall also make use of the identity for all $u,v\in H^1(\Omega)$, $\bw\in H^1(\Omega)^d$:
\begin{equation}\label{int_bp}
(\bw\cdot\nabla u,v)+\frac12((\Div \bw)u,v)=\frac12\left((\bw\cdot\nabla u, v)-(\bw\cdot\nabla v, u)\right)+ \frac12\int_{\dO_0}(\bn\cdot\bw)uv\,\mathrm{d}s.
\end{equation}

We multiply the first equality in \eqref{NSE} by $J\bu$, integrate it over the reference domain,
and employ \eqref{int_bp} for integration by parts. We get
\begin{multline*}
\frac12\frac{d}{dt} \|J^{\frac12}\bu\|^2-\frac{1}2 (J_{t}\,\bu,\bu) +
(J(\hat\bsigma\circ\bxi)\bF^{-T},\nabla \bu)-\int_{\dO_0}\left(J(\hat\bsigma\circ\bxi)\bF^{-T}\bn\right)\cdot{\bxi}_{t}\,\mathrm{d}s\\ +
\frac{1}2 (\Div(J\bF^{-1}(\bu-{\bxi}_{t}))\,\bu,\bu)=(J\blf,\bu),
\end{multline*}
here $\bn$ is the exterior unit normal vector on $\dO_0$.
The mass balance yields the equality
\begin{equation}\label{FSIaux}
{J}_{t}+\Div(J\bF^{-1}(\bu-{\bxi}_{t}))=0\quad\text{in}~Q.
\end{equation}
This identity leads to some cancellations and we get
\begin{equation*}
\frac12 \frac{d}{dt}\|J^{\frac12}\bu\|^2 +
(J(\hat\bsigma\circ\bxi)\bF^{-T},\nabla \bu)-\int_{\dO_0}\left(J(\hat\bsigma\circ\bxi)\bF^{-T}\bn\right)\cdot{\bxi}_{t}\,\mathrm{d}s=\left(J\blf,\bu\right).
\end{equation*}
The Piola identity, $\Div(J\bF^{-1})=0$, implies the following equality
\begin{equation}\label{aux_piola}
\Div(J\bF^{-1}\bu)=J(\nabla \bu):\bF^{-T}\quad\text{in}~Q ,
\end{equation}
where $\bA:\bB:=\mbox{tr}(\bA\bB^T)$.
Using the notation $\bD_{\xi}(\bu)=\frac12(\nabla\bu \bF^{-1}+\bF^{-T}(\nabla\bu)^T)$ for the rate of deformation tensor in the reference coordinates, we get with the help of \eqref{aux_piola} and the second equation in \eqref{NSE}
\[
\begin{split}
(J(\hat\bsigma\circ\bxi)\bF^{-T},\nabla \bu)&= (J(-p\bI+\nu(\nabla\bu \bF^{-1}+\bF^{-T}(\nabla\bu)^T))\bF^{-T},\nabla \bu)
=2\nu (J\bD_{\xi}(\bu)\bF^{-T},\nabla \bu)\\&=2\nu (J\bD_{\xi}(\bu),\nabla \bu\bF^{-1}) =
2\nu (J\bD_{\xi}(\bu),\bD_{\xi}(\bu)).
\end{split}
\]
In the last equality we used that for any symmetric tensor $\bA$ and any tensor $\bB$, it holds $\bA:\bB=\frac12\bA:(\bB+\bB^T)$.
Therefore, the energy balance equality in reference coordinates takes the form
\begin{equation} \label{Energy}
\frac12 \frac{d}{dt}\|J^{\frac12}\bu\|^2+
2\nu\|J^{\frac12}\bD_{\xi}(\bu)\|^2-\int_{\dO_0}\left(J(\hat\bsigma\circ\bxi)\bF^{-T}\bn\right)\cdot{\bxi}_{t}\,\mathrm{d}s =(J\blf,\bu)\,.
\end{equation}
The mechanical interpretation of \eqref{Energy} is the following one: the work of  external forces (right-hand side) is balanced by the change of kinetic energy  (the first term),  viscous
dissipation of energy (the second term), and flow intensification due to the boundary condition (the third term).

\subsection{Weak formulation}
For $t\in[0,T]$ we  introduce the following time-dependent trilinear and bilinear forms:
\begin{align*}
c(\bxi;\bw,\bu,\bpsi)&=\int_{\Omega_0}  J\left((\nabla\bu)\bF^{-1}
\bw\right)\cdot \bpsi\,\mathrm{d}\bx, \quad \bw,\bu,\bpsi \in H^1(\Omega_0)^d,\\
a(\bxi;\bu,\bpsi)&=\int_{\Omega_0}2\nu  J\bD_\xi\bu:\bD_\xi\bpsi\,\mathrm{d}\bx, \quad \bu,\bpsi \in H^1(\Omega_0)^d,\\
b(\bxi;p,\bpsi)&=
\int_{\Omega_0}p  J\bF^{-T}:\nabla \bpsi\,\mathrm{d}\bx, \quad p\in L^2(\Omega_0),~\bpsi \in H^1(\Omega_0)^d,
\end{align*}
\rev{with $J=J(\bxi)$, $\bF=\bF(\bxi)$.}

The weak formulation of \eqref{NSE}--\eqref{constit_f} reads: Find $\{\bu,p\}\in L^2(0,T;H^1(\Omega_0)^d)\cap L^\infty(0,T;L^2(\Omega_0)^d)\times L^2(Q)$  satisfying $\bu=\bxi_t$ on $\dO_0^{ns}$, $\bu=\bu_D$ on $\dO_0^{D}$ and
\begin{equation}\label{weak}
\left(  J \bu_t,\bpsi\right) +c({\bxi};\bu-{\bxi}_t,\bv,\bpsi)+a({\bxi};\bu,\bpsi)  -b({\bxi};p,\bpsi)+b({\bxi};q,\bu)=(J {\blf},\bpsi)+\int_{\dO_0^N}J\bg\cdot\bpsi{\rm d}s
\end{equation}
for all $\bpsi\in H^1(\Omega_0)^d,$ $\bpsi=0$ on $\dO_0^{ns}\cup\dO_0^{D}$, $q\in L^2(\Omega_0)$ for all $t\in[0,T]$.  

\section{Discretization method}\label{sec_FE}
In this section we introduce both time and space discretizations of
 the formulation \eqref{NSE} in the reference domain.
Treating the flow problem in reference coordinates allows us to avoid triangulations and finite element function spaces
dependent on time. In this paper,  we assume that the mapping $\bxi$ is given explicitly and used in the finite element formulation without any further numerical approximation apart from the boundary condition.

Let  a collection of {simplices} $\mathcal{T}_h$ (triangles for $d=2$ and tetrahedra for $d=3$) form a consistent regular triangulation $\mathcal{T}_h$ of the reference domain $\overline{\Omega}_0$. We let $h=\max_{T\in\mathcal{T}_h}\text{diam}(T)$. Consider conforming FE spaces $\VV_h\subset H^1(\Omega_0)^d$ and $\QQ_h\subset L^2(\Omega_0)$; $\VV_h^0$ is a subspace of $\VV_h$ of functions vanishing on $\dO_0^{ns}\cup\dO_0^{D}$.
 We assume that $\VV_h^0$ and $\QQ_h$ form the LBB-stable  finite element pair: There exists a mesh-independent constant $c_0$, such that
\begin{equation}
\inf_{q_h\in \QQ_h}\sup_{\bv_h\in \VV_h^0}\frac{(q_h,\Div\bv_h)}{\|\nabla\bv_h\|\|q_h\|}
\ge c_0>0.
\label{LBB}
\end{equation}
As an example of admissible discretization, we consider the generalized Taylor-Hood finite element spaces,
\begin{equation}\label{defVQ}
\begin{aligned}
\VV_h&=\{ \bu_h \in C(\Omega_0)^d\,:\, \bu_h|_T \in \left[P^{m+1}(T)\right]^d, \forall~ T \in \mathcal{T}_h\}, \\
\QQ_h&=\{ q_h \in C(\Omega_0)\,:\, q_h|_T \in P^{m}(T), \forall~ T \in \mathcal{T}_h\},
\end{aligned}
\end{equation}
where integer $m\ge1$ is polynomial degree.

Assuming a constant time step $\Delta t=\frac{T}{N}$,  we use the notation $\bu^k(\bx):=\bu(k\Delta t, \bx)$, and similar
for $p$ and $\bxi$. To emphasize the dependence on $k$, denote $\bF_k:=\nabla{\bxi}^k$, $J_k:=\mbox{det}(\bF_k))$, $\bD_k(\bv):=\bD_{\xi^k}(\bv)$.

For given spatial functions $f^i$, $i=0,\dots,k$,  $\left[ f\right]^{k}_t:=\frac{f^k-f^{k-1}}{\Delta t}$
denotes the backward finite difference at $t_k=k\Delta t$. For a sufficiently smooth vector function $\bv$, denote by $I_h(\bv)\in\VV_h$ its \emph{nodal Lagrange interpolant}.

Let $\bu_h^0=I_h(\bu(t_0))$. The finite element discretization of \eqref{weak} reads: For $k=1,2,\dots$, find $\{\bu^{k}_h,p^{k}_h\}\in \VV_h\times \QQ_h$ satisfying  $\bu^{k}_h=I_h(\bxi_t^k)$ on $\dO_0^{ns}$,  $\bu^{k}_h=I_h(\bu_D^k)$ on $\dO_0^{D}$ and the following equations
\begin{multline}\label{FE1}
\left(  J_{k-1}\left[\bu_h\right]_t^{k},\bpsi_h\right)+\left(\frac{1}2\left[J\right]^{k}_t\bu^{k}_h,\bpsi_h\right)
+\frac{1}2(\Div\big(J_{k}\bF^{-1}_{k}\bw^k_h\big)\bu^{k}_h,\bpsi_h)\\ +c({\bxi}^{k};\bw^{k}_h,\bu^{k}_h,\bpsi_h)
+a({\bxi}^{k};\bu^{k}_h,\bpsi_h) -b({\bxi}^{k};p^{k}_h,\bpsi_h)+b({\bxi}^{k};q_h,\bold{u}^{k}_h)=(J_k {\blf}^{k},\bpsi_h)+\int_{\dO_0^N}J_k\bg^k\cdot\bpsi{\rm d} s
\end{multline}
for all $\bpsi_h\in\VV_h^0,$ $q_h\in\QQ_h$ with advection velocity $\bw^{k}_h:=\big({\bu_h^{k-1}-\bxi^k_t}\big)$.
\smallskip

The second and the third terms in \eqref{FE1} are consistent due to the identity \eqref{FSIaux} and are added in the FE formulation
to enforce the conservation property  of the discretization. While our computations show that in practice \rev{these terms}
can be skipped, we need these terms for the stability bound in the next section. In the numerical analysis of incompressible Navier-Stokes equations in the Eulerian description, including these terms corresponds to the Temam's skew-symmetric form of the convective terms~\cite{Temam}.

Note that the inertia terms are linearized so that a \emph{linear} algebraic system should be solved on each time step. In the next section we show that the finite element method is energy stable.

Note that if $\dO_0^N=\emptyset$, then the boundary condition should be compatible with divergence constraint. Therefore, in this case  we let $c_{\bot}^k:=\int_{\dO_0}J_k\bF^{-T}_kI_h(\bu^k)\cdot\bn \mathrm{d} s$ and let $\bu^{k}_h=I_h(\bxi_t^k)-c_{\bot}^k$ on $\dO_0^{ns}$ and $\bu^{k}_h=I_h(\bu_D(t_k))-c_{\bot}^k$ on $\dO_0^{D}$.

\begin{remark}\rm Condition $\bu_h=0$ on those parts of the \rev{boundary}, where no-slip and no-penetration takes place,  of an evolving fluid domain is not physically reasonable. Moreover, it leads to strong simplifications of finite element analysis. Indeed, we shall see that constructing suitable FE extensions of boundary conditions to computational domain is not straightforward.
\end{remark}

\section{Stability of FEM solution}\label{s_stab}
From now on we assume $\dO_0^{ns}=\dO_0$ for all $t\in[0,T]$.
To show the stability, we need several preparatory steps which help us to handle non-homogeneous boundary conditions
and time-dependent bilinear forms.
First we note the  Korn's-type inequality in the reference domain,
\begin{equation}\label{Korns}
\|\nabla \bu\|\le C_K \|J^{\frac12}\bD_{\xi}(\bu)\|\quad\forall~\bu\in H^1_0(\Omega_0)^d,\quad t\in[0,T],
\end{equation}
with $C_K$ uniformly bounded with respect to $t\in[0,T]$. The estimate \eqref{Korns} easily follows from the standard Korn inequality and assumptions in \eqref{assumption}; see~\cite{DLOVpaper}.
Thanks to \eqref{Korns}  the bilinear form $a(\bxi(t);\cdot,\cdot)$
is coercive  on $H^1_0(\Omega_0)^d\times H^1_0(\Omega_0)^d$ uniformly in time.
Also due to \eqref{assumption} the bilinear forms $a(\bxi(t);\cdot,\cdot)$ and  $b(\bxi(t);\cdot,\cdot)$ are
continuous uniformly in time on $H^1(\Omega_0)^d\times H^1(\Omega_0)^d$ and $L^2(\Omega_0)\times H^1(\Omega_0)^d$, respectively.

Unlike the continuous case, the finite element solution $\{\bu^{k}_h,p^{k}_h\}$ does not satisfy a strong formulation
and so the arguments from section~\ref{s_energy} do not apply directly. To show the proper energy balance for
the finite element solution, we split it into a part vanishing on the boundary and another \textit{a priori} defined (and so stable) part, which satisfy correct boundary conditions. Therefore, we consider  decomposition $\bu_h^k=\bv_h^k+\bv_{h,1}^k$, such that
\begin{equation}\label{decomp}
\bv_{h,1}^k=\bu_h^k~~\text{on}~ \dO_0\times[0,T],\quad b({\bxi}^{k};q_h,\bv_{h,1})=0~~ \forall~q_h\in\QQ_h,~~ k=1,2,\dots,N,
 \end{equation}
and
\begin{equation}\label{stab_apr}
\|\bv_{h,1}^k\|_{W^{1,\infty}}\le C, \quad (J_{k-1}(\bv_{h,1}^k-\bv_{h,1}^{k-1}),\bv_h^k)\le C\Delta t( \|\bv_h^k\|+ h^{m+2}\|\left[\bv_h\right]^{k}_t\|),
\end{equation}
with some real $C$ depending only on data and independent of $k$, $h$.
 The existence of such decomposition will be explicitly demonstrated in the next section.

With the help of this decomposition, the finite element method can be re-formulated as follows: find $\{\bv^{k}_h,p^{k}_h\}\in \VV_h^0\times \QQ_h$ satisfying
for all $\bpsi_h\in\VV_h^0,$ $q_h\in\QQ_h$
\begin{multline}\label{FE1hom}
\left(  J_{k-1}\left[\bv_h\right]_t^{k},\bpsi_h\right)+\left(\frac{1}2\left[J\right]^{k}_t\bv^{k}_h,\bpsi_h\right)
+\frac{1}2(\Div\big(J_{k}\bF^{-1}_{k}\bw^k_h\big)\bv^{k}_h,\bpsi_h)\\ +c({\bxi}^{k};\bw^{k}_h,\bv^{k}_h,\bpsi_h)
+c({\bxi}^{k};\bv^{k}_h,\bv_{h,1}^k,\bpsi_h)
+a({\bxi}^{k};\bv^{k}_h,\bpsi_h) \\ -b({\bxi}^{k};p^{k}_h,\bpsi_h)+b({\bxi}^{k};q_h,\bold{v}^{k}_h)=\la\widetilde{\blf}^{k}_h,\bpsi_h\ra
\end{multline}
with  $\la\widetilde{\blf}^{k}_h,\bpsi_h\ra= ( J_k(\blf-(\nabla \bv_{h,1}^k)\bF^{-1}(\bv_{h,1}^k-{\bxi}_{t}^k)),\bpsi_h)- a({\bxi}^{k};\bv_{h,1}^k,\bpsi_h)- (J_{k-1}\left[\bv_{h,1}\right]_t^{k},\bpsi_h)  $.
We stress that the formulation \eqref{FE1hom} appears here only for the purpose of analysis. Although equivalent to \eqref{FE1}, the formulation \eqref{FE1hom} is not practical, since it requires an explicit knowledge of $\bv_{1,h}$. For the analysis, it is sufficient to know that $\bv_{1,h}$ exists.

We test \eqref{FE1hom} with $\bpsi_h=\bv^{k}_h$, $q_h=p_h^k$.
We handle each resulting term separately and start with the first term in \eqref{FE1hom}:
\begin{multline}\label{int_t}
(J_{k-1}\left[\bv_h\right]^{k}_t,\bv_h^k)=
\frac{1}{2\Delta t}\left(\|J_{k}^{\frac12}\bv^{k}_h\|^2-\|J_{k-1}^{\frac12}\bv^{k-1}_h\|^2\right)\\
-\frac12(\left[J\right]^{k}_t\bv^{k}_h,\bv^{k}_h)
+\frac{\Delta t}2\|J_{k-1}^{\frac12}\left[\bv_h\right]^{k}_t\|^2\,.
\end{multline}
The term $-\frac12(\left[J\right]^{k}_t\bv^{k}_h,\bv^{k}_h)$ in \eqref{int_t} cancels with the second term in \eqref{FE1hom}.
Applying \eqref{int_bp} to the fourth (inertia) term in \eqref{FE1hom} and using boundary conditions give
\begin{equation}\label{int_inert}
(J_k(\nabla\bv^{k}_h\bF^{-1}_k\bw^k_h),\bv^k_h)=
-\frac{1}{2}(\Div\left(J_k\bF^{-1}_k \bw^k_h\right)\bv^{k}_h,\bv^{k}_h).
\end{equation}
This cancels with the third term in \eqref{FE1hom}. We keep the fifth term as it is.
The sixth term in \eqref{FE1hom} gives
\[
a({\bxi}^{k};\bv^{k}_h,\bv^{k}_h)=
2\nu  \left(J_k\bD_k(\bv^{k}_h),\bD_k(\bv^{k}_h)\right)=2\nu  \left\|J_{k}^{\frac12}\bD_{k}(\bv^{k}_h)\right\|^2\,.
\]
The $b$-terms cancel out for $q_h=p_h^k$.
Substituting all equalities back into \eqref{FE1hom}, we obtain
the following energy balance for the $\bv_h$-part of finite element solution $\bu_h$:
\begin{multline}\label{balance}
{\frac{1}{2\Delta t}\left(\|J_{k}^{\frac12}\bv^{k}_h\|^2-\|J_{{k-1}}^{\frac12}\bv^{k-1}_h\|^2\right)}
+2\nu  \left\|J_{k}^{\frac12}\bD_{k}(\bv^{k}_h)\right\|^2
+\frac{\Delta t}2\left\|J^{\frac12}_{k-1}\left[\bv_h\right]_t^{k}\right\|^2 \\
+(J_k (\nabla \bv_{h,1}^k\bF^{-1}_k)\bv_h^k,\bv_h^k)=\la\widetilde{\blf}_h^{k},\bv^k_h\ra.
\end{multline}
We deduce an energy stability estimate for the finite element method from the balance in \eqref{balance} and \textit{a priori} estimates in \eqref{stab_apr}.
For the sake of notation, we introduce  $\|\cdot\|_{k}:=\left(\int_{\Omega_0} J_{k}|\cdot|^2\,\mathrm{d}\bx\right)^{\frac12}$,
which defines a $k$-dependent norm uniformly equivalent to the $L^2$-norm.
Using estimates  \eqref{stab_apr} and the definition of $\widetilde{\blf}_h^{k}$ one shows that the forcing
term is bounded,
\[
\begin{split}
\la\widetilde{\blf}_h^{k},\bv^k_h\ra &\le C(\|\blf\|\|\bv^k_h\|+\|\nabla \bv_{h,1}^k\|(\|\bv_{h,1}^k\|_{L^\infty}+\|{\bxi}_{t}\|)\|\bv^k_h\|+\|\nabla\bv_{h,1}^k\|\|\bD_{k}\bv^k_h\|)\\
&\quad+ C( \|\bv_h^k\|+ h^{m+2}\|\left[\bv_h\right]^{k}_t\|)\\
&\le C (\|\bD_{k}(\bv^{k}_h)\|_{k}+(\Delta t)^{-1}h^{2m+4})  +\frac{\Delta t}2\left\|\left[\bv_h\right]_t^{k}\right\|^2_{k-1}
\end{split}
\]
with a constant $C$ depending only on problem data.  We substitute this in   \eqref{balance} and further use the bound $C \|\bD_{k}(\bv^{k}_h)\|_{k}\le \nu^{-1}C^2+\nu\|\bD_{k}(\bv^{k}_h)\|_{k}^2$. This yields the estimate
\begin{equation}\label{balance1}
{\frac{1}{2\Delta t}\left(\|\bv^{k}_h\|^2_k-\|\bv^{k-1}_h\|^2_{k-1}\right)}
+\nu  \left\|\bD_{k}(\bv^{k}_h)\right\|^2_k
+(J_k (\nabla \bv_{h,1}^k\bF^{-1}_k)\bv_h^k,\bv_h^k)\le C(\rev{\nu^{-1}}+(\Delta t)^{-1}h^{2m+4})) ,
\end{equation}
where the constant $C$ depends on the problem data, but not on $\nu$, $h$ and $\Delta t$.

Thanks to Sobolev's  {embedding} inequalities {as well as}  \eqref{assumption} and \eqref{Korns}, we bound the fourth term in \eqref{balance1} resulting from the boundary motion in two ways,
\[
|(J_k (\nabla \bv_{h,1}^k\bF^{-1}_k)\bv_h^k,\bv_h^k)|\le
\left\{\begin{aligned}&
C\|\nabla \bv_{h,1}^k\|\|\nabla \bv_h^k\|^2\le C_1\|\nabla \bv_{h,1}^k\|\|\bD_k(\bv_h^k)\|^2_k, \\
&C\|\nabla \bv_{h,1}^k\|_{L^\infty(\Omega_0)}\|\bv_h^k\|^2\le C_2\|\bv_h^k\|^2_k.
\end{aligned}\right.
\]
If the factor $C_1\|\nabla \bv_{h,1}^k\|$ is not too large such that it holds
\begin{equation}\label{cond1}
C_1\|\nabla \bv_{h,1}^k\|\le\nu/2,
\end{equation}
then the intensification term can be absorbed by the viscous dissipation term. \rev{Note that $C_1$ depends on $\Omega_0$ and $\bxi$.}
So we obtain from \eqref{balance1}
\begin{equation}\label{aux2}
\frac{1}2\|\bv^{k}_h\|_{k}^2
+\nu \Delta t \|\bD_{k}(\bv^{k}_h)\|_{k}^2
\le \frac{1}2\|\bv^{k-1}_h\|_{{k-1}}^2+ C\Delta t(\rev{\nu^{-1}}+(\Delta t)^{-1}h^{2m+4})),
\end{equation}
with some $C$ depending only on problem data.

Summing up inequalities \eqref{aux2} over $k=1,\dots,n$, $n\le N$, gives
\begin{equation}\label{est_energy}
\frac{1}2\|\bv^{n}_h\|_{{n}}^2+
\nu  \sum_{k=1}^{n}\Delta t\|\bD_{k}(\bv^{k}_h)\|_{{k}}^2
\le
\frac{1}2\|\bv_{0}\|_{{0}}^2
+C\Delta t(\rev{\nu^{-1}}+(\Delta t)^{-1}h^{2m+4})).
\end{equation}
Otherwise, if \eqref{cond1} does not hold, we estimate
\begin{equation}\label{aux2a}
\frac{1}2(1-2C_2\Delta t)\|\bv^{k}_h\|_{k}^2
+\nu \Delta t \|\bD_{k}(\bv^{k}_h)\|_{k}^2
\le \frac{1}2\|\bv^{k-1}_h\|_{{k-1}}^2+ C\Delta t(\rev{\nu^{-1}}+(\Delta t)^{-1}h^{2m+4})).
\end{equation}
Now we assume that $\Delta t$ is sufficiently small such that $(1-2C_2\Delta t)=\alpha>0$. Summing over $k = 1,\dots,n$, gives
\begin{equation}\label{aux3a}
\frac{1}2\|\bv^{n}_h\|_{n}^2
+\nu \sum_{k=1}^{n}\Delta t \|\bD_{k}(\bv^{k}_h)\|_{k}^2
\le C_2\sum_{k=1}^{n}\Delta t\|\bv^{k}_h\|_{k}^2 + \frac{1}2\|\bv_{0}\|_{{0}}^2+ C(\rev{\nu^{-1}}+(\Delta t)^{-1}h^{2m+4}).
\end{equation}
Applying discrete Gronwall's inequality yields
\begin{equation}\label{est_energya}
\frac{1}2\|\bv^{n}_h\|_{{n}}^2+
\nu \sum_{k=1}^{n}\Delta t\|\bD_{k}(\bv^{k}_h)\|_{{k}}^2
\le
e^{\frac{2C_2}{\alpha}T}\left(\frac{1}2\|\bv_{0}\|_{{0}}^2
+C(\rev{\nu^{-1}}+(\Delta t)^{-1}h^{2m+4}))\right).
\end{equation}
\rev{Constant $C_2$ depends on $\bxi$, which defines the evolution of $\Omega(t)$, but not on $\nu$.}

Finally, we apply the triangle inequality and \eqref{stab_apr} one more time to show the stability bound for the velocity solution $\bu_h$ to \eqref{FE1},
\begin{equation}\label{Stab_u}
\frac{1}2\|\bu^{n}_h\|_{{n}}^2+
\nu \sum_{k=1}^{n}\Delta t\|\bD_{k}(\bu^{k}_h)\|_{{k}}^2
\le C(\rev{\nu^{-1}}+(\Delta t)^{-1}h^{2m+4})
\le C_{\rm stab},\quad\text{if}~~c \Delta t\ge h^{2m+4},
\end{equation}
where constant $C_{\rm stab}$ depends on the problem data, i.e. $\blf$, $T$, $\Omega_0$, $\bxi$, $\rev{\nu}$, but not on $h$ and $\Delta t$. If condition \eqref{cond1} does not hold, then $\Delta t$ is assumed sufficiently small to satisfy $1-2C_2\Delta t>0$.

Note that the restriction $\Delta t\ge h^{2m+4}$ has resulted from handling non-homogeneous boundary conditions.
Indeed, one gets the extra $h$-dependent term in \eqref{stab_apr} while estimating the difference $\bv_{h,1}^k-\bv_{h,1}^{k-1}$.
If one defines $\bv_{h,1}^k$ as a nodal interpolant to the smooth divergence-free function $\bv_1$
then the estimate as in \eqref{stab_apr} follows without this extra term. However, the nodal interpolant does not satisfy the ``divergence-free'' condition in \eqref{decomp} and $\bv_{h,1}^k$ is defined as a suitable projection of $\bv_1$.

\section{Convergence analysis}\label{sec_anal} In this section, we demonstrate optimal order convergence of $\bu_h$ to $\bu$.
From now on,  we assume that $\Omega_0$ is a convex polyhedron (polygon if $d=2$).

\subsection{Preliminaries}
For the Navier-Stokes equations solution in the reference domain we recall the notations $\bu^k:=\bu(k\Delta t),p^k:=p(k\Delta t)$.
  The finite element solution to \eqref{FE1} is $\{\bu^k_h,p^k_h\}$. Furthermore,
  \[\{\be^k, e^k\}:= \{\bu^k-\bu^k_h,p^k-p^k_h\}\]
   denotes  the finite element error.


For the \textit{fixed} time instance $t_k=k \Delta t$, we define the subspace of $\VV_h$,
\[
\XX_h^k:=\{\bpsi_h\in\VV_h\,:\,b(\bxi^k;q_h,\bpsi_h)=0~~\forall~q_h\in\QQ_h\}.
\]
In  Lemma~\ref{L1} below we prove an important technical result: if the pair of spaces $\VV_h-\QQ_h$ is inf-sup stable in the sense of \eqref{LBB},
and $\bxi$ satisfies certain assumptions, then the subspace $\XX_h^k$ has full approximation properties.
Here and further we assume that for all $t\in[0,T]$ the domains $\Omega(t)$ are such that the Stokes problem in the physical domain $\Omega(t)$ satisfies  $H^2$-regularity property:
for any $\blf\in L^2(\Omega(t))$  the unique solution $\{\hat\bphi,\hat\lambda\}$ of
\begin{equation}\label{Stokes}
-\Delta\hat\bphi+\nabla \hat\lambda=\blf,\quad\Div\hat\bphi=0~~\text{in}~\Omega(t),\quad\hat\bphi=0~~\text{on}~\dO(t)
\end{equation}
satisfies $\hat\bphi\in H^2(\Omega(t))^d$, $\hat\lambda\in H^1(\Omega(t))$ and
\begin{equation}\label{eqH2}
\|\hat\bphi\|_{H^2(\Omega(t))}+\|\hat\lambda\|_{H^1(\Omega(t))}\le C_R\|\blf\|_{L^2(\Omega(t))},
\end{equation}
with some $C_R$ uniformly bounded for $t\in[0,T]$.

We recall that $I_h(\bv)\in\VV_h$ denotes the nodal interpolant of $\bv$  and  $c_\bot:=\int_{\dO_0}J\bF^{-T}I_h(\bv)\cdot\bn\,{\rm d}s$.

\begin{lemma}\label{L1} Assume \eqref{LBB}  and
\begin{equation}\label{assum2}
\sup_{Q}\|\bI-\bF\|_F\le \eps,
\end{equation}
with sufficiently small $\eps>0$ and the identity matrix $\bI$. Then for $\bv\in H^{m+\frac52}({\Omega}_0)^d$ satisfying $\Div(J_k\bF^{-1}_k\bv)= 0$, $1\le k\le N$, we have
\begin{equation}\label{etas}
\inf_{\overset{\bv_h\in\XX_h^k}{\bv_h=I_h(\bv)-c_\bot~\text{on}~\dO_0}}(\|\bv-\bv_h\|+h\|\nabla(\bv-\bv_h)\|) \le c\,h^{m+2}\|\bv\|_{H^{m+\frac52}(\Omega_0)},
\end{equation}
where $m$ is the polynomial degree from \eqref{defVQ} and $c$ depends on $\bxi$, but not on $k$, $h$ or $\bv$.
\end{lemma}
\begin{proof} From \eqref{assum2} it follows that
\begin{equation}\label{aux3}
\sup_{Q}\|\bI-\bF^{-1}\|_F\le \frac{\eps}{1-\eps},\quad \sup_{Q}|1-J|\le c\,\eps
\end{equation}
holds with some $c>0$ depending only on space dimension $d=2,3$. The inf-sup condition \eqref{LBB} implies that for a given $q_h\in\QQ_h$
there exists $\bpsi_h\in\VV_h^0$ such that $(q_h,\Div \bpsi_h)=\|q_h\|^2$ and $c_0\|\nabla \bpsi_h\|\le\|q_h\|$. Thanks to
  \eqref{aux3} we estimate
\[
\begin{aligned}
b(\bxi^k;q_h,\bpsi_h)&=(J_k\bF^{-T}_k:\nabla \bpsi_h, q_h)=(\Div\bpsi_h, q_h)+([J_k\bF^{-T}_k-\bI]:\nabla \bpsi_h, q_h)\\
&\ge\|q_h\|^2-\sup_{Q}\|\bI-\bF^{-1}J\|_F\|\nabla \bpsi_h\|\|q_h\|\\
&\ge \|q_h\|^2 -(\sup_{Q}\|\bI-\bF^{-1}\|_F+c\eps\sup_{Q}\|\bF^{-1}\|_F) \|\nabla \bpsi_h\|\|q_h\|\\
&\ge c_0\|\nabla \bpsi_h\|\|q_h\|-\tilde c\eps\|\nabla \bpsi_h\|\|q_h\|= (c_0-\tilde c\eps)\|\nabla \bpsi_h\|\|q_h\|.
\end{aligned}
\]
Thus we proved that for $\eps>0$ small enough, the bilinear form $b(\bxi^k;\cdot,\cdot)$ satisfies the uniform inf-sup condition
\begin{equation}
\inf_{q_h\in \QQ_h}\sup_{\bpsi_h\in \VV_h^0}\frac{b(\bxi^k;q_h,\bpsi_h)}{\|\nabla\bpsi_h\|\|q_h\|}
\ge \hat{c}_0>0,
\label{LBBk}
\end{equation}
with $\hat{c}_0$ independent of $k$ and $h$.

Now consider arbitrary $\bv\in H^{m+\frac52}(\overline{\Omega}_0)^d$ satisfying $\Div(J\bF^{-1}\bold{v})= 0$, and hence $b(\bxi^k;q,\bv)=0$ for all $q\in L^2(\Omega_0)$.
The nodal Lagrange interpolant of $\bv$  is well defined.   From $\int_{\dO_0}J\bF^{-1}\bv\cdot\bn\mathrm{d}s=0$ and approximation properties
of the polynomial interpolation, we get
  \begin{equation}\label{bc_est0}
\begin{split}
|c_\bot|&= |\int_{\dO_0}J\bF^{-1}(I_h(\bv)-\bv)\cdot\bn{\rm d}s|\le c\|I_h(\bv)-\bv\|_{L^2(\dO_0)}= c\sum_{\Gamma_i\subset\dO_0}\|I_h(\bv)-\bv\|_{L^2(\Gamma_i)} \\
&\le ch^{m+2}\sum_{\Gamma_i\subset\dO_0}\|\bv\|_{H^{m+2}(\Gamma_i)}\le ch^{m+2}\|\bv\|_{H^{m+\frac52}(\Omega_0)},
\end{split}
\end{equation}
where we sum over all faces $\Gamma_i$ of our polyhedral domain $\Omega_0$. 
By the same argument and also using the triangle inequality, we get
\begin{equation}\label{bc_est}
\|\bv-(I_h(\bv)-c_\bot)\|_{L^2(\dO_0)}\le ch^{m+2}\|\bv\|_{H^{m+\frac52}(\Omega_0)}.
\end{equation}
Now we define $\bv_h$ as the quasi-Stokes projection of $\bv$ satisfying suitable boundary condition: One finds $\bv_h\in\VV_h$ such that $\bv_h=I_h(\bv)-c_\bot$ on $\dO_0$
\[
a(\bxi^k;\bv-\bv_h,\bpsi_h)-b(\bxi^k;p_h,\bpsi_h)+b(\bxi^k;q_h,\bv_h)=0\quad\forall~\bpsi_h\in\VV_h^0,~q_h\in\QQ_h.
\]
For zero Dirichlet boundary conditions, showing the optimal order estimate of $\|\nabla(\bv-\bv_h)\|$ for the Stokes projection is standard, see, e.g., Theorem 5.2.1--5.2.2 in \cite{boffi2013mixed} or Section~2.4 in \cite{ern2013theory}.
To handle non-homogenous boundary conditions, we use Proposition~8 from \cite{gunzburger1992treating}. The proof of the proposition in \cite{gunzburger1992treating} is given for the Stokes problem, but the arguments  only need the coercivity  of the bilinear form $a(\bxi^k;\cdot,\cdot)$  on $H^1_0(\Omega_0)^d\times H^1_0(\Omega_0)^d$, the continuity of $a(\bxi^k;\cdot,\cdot)$  and $b(\bxi^k;\cdot,\cdot)$ and the inf-sup property \eqref{LBBk}. This proposition establishes the estimate
\[
\|\nabla(\bv-\bv_h)\| \le C\, \inf_{\overset{\psi_h\in\XX_h^k}{\psi_h=I_h(\bv)-c_\bot~\text{on}~\dO_0}}\|\bv-\bpsi_h\|_{H^1(\Omega_0)}.
\]
Letting   $\bpsi_h=I_h(\bv)-c_\bot$ on the right-hand side of this inequality, using \eqref{bc_est0} and interpolation properties of polynomial functions, we find the bound
\[
\|\nabla(\bv-\bv_h)\| \le ch^{m+1}\|\bv\|_{H^{m+\frac52}(\Omega_0)}.
\]

In order to show the higher order bound for $\bv-\bv_h$ in the $L^2$ norm, we consider  the weak formulation of the quasi-Stokes problem with $J_k(\bv-\bv_h)$ on the right-hand side:
Find  $\bphi$, $\lambda$ such that
\[
a(\bxi^k;\bphi,\bpsi)-b(\bxi^k;\lambda,\bpsi)+b(\bxi^k;q,\bphi)=(J_k(\bv-\bv_h),\bpsi)\quad\forall~\bpsi\in H^1_0(\Omega_0)^d,~q\in L^2(\Omega_0).
\]
Denote $\hat\bphi=\bphi\circ(\bxi^k)^{-1}$, $\hat\lambda=\lambda\circ(\bxi^k)^{-1}$, then $\{\hat\bphi,\hat\lambda\}$ satisfy \eqref{Stokes}--\eqref{eqH2}
with $t=k\Delta t$, $\blf=(\bv-\bv_h)\circ(\bxi^k)^{-1}$. From this and \eqref{assumption}, we derive
\[
\|\bphi\|_{H^2(\Omega_0)}+\|\lambda\|_{H^1(\Omega_0)}\le C\|\bv-\bv_h\|.
\]
Now the bound
$\|\bv-\bv_h\|\le C\,h^{m+2}\|\bv\|_{H^{m+\frac52}(\Omega_0)}$ follows by the generalized Nitsche argument of Proposition~9 from
~\cite{gunzburger1992treating} and \eqref{bc_est}.
\end{proof}

 Note that the extra $\frac12$ in the Sobolev space order arises in the Lemma due to the trace theorem needed to treat non-homogeneous boundary conditions.

\begin{remark}
  \rm In the statement of the lemma and in further arguments one may replace $H^{m+\frac52}(\Omega_0)$ by  $C^{m+2}(\overline{\Omega}_0)$ . The only change required in the proof is a slightly different argument to show the estimate in \eqref{bc_est0} and \eqref{bc_est}. 
  From $\int_{\dO_0}J\bF^{-1}\bv\cdot\bn\mathrm{d} s=0$ and approximation properties
of polynomial interpolation, we get
\begin{equation}\label{bc_est0b}
\begin{split}
|c_\bot|= &|\int_{\dO_0}J\bF^{-1}(I_h(\bv)-\bv)\cdot\bn{\rm d}s|\le c\|I_h(\bv)-\bv\|_{L^2(\dO_0)}\le \\ &\le ch^{m+2}\|\bv\|_{C^{m+2}(\dO_0)}\le ch^{m+2}\|\bv\|_{C^{m+2}(\overline{\Omega}_0)}.
\end{split}
\end{equation}
\end{remark}

\vskip 5mm

By $P_h^k\,:\,C(\overline{\Omega}_0)^d\to\XX_h^k$ we denote the  projection with respect to the scalar product $(J_{k}\cdot,\cdot)$, defined by $$\|\bv-P_h^k(\bv)\|_{k-1}=\inf_{\overset{\psi_h\in\XX_h^k}{\psi_h=I_h(\bv)-c_\bot~\text{on}~\dO_0}}\|\bv-\bpsi_h\|_{k-1}.$$
With the help of the standard variational argument one checks the orthogonality property
\begin{equation}\label{Orth}
(J_{k-1}(\bv-P_h^k(\bv)),\bpsi)=0\quad\forall~\bpsi\in\XX_h^k\cap H^1_0(\Omega)^d.
\end{equation}
Due to the equivalence of $L^2$ and $\|\cdot\|_{k-1}$ norms, we have
\begin{equation}\label{aux4}
\|\bv-P_h^k(\bv)\|\le  c\,\inf_{\overset{\psi_h\in\XX_h^k}{\psi_h=I_h(\bv)-c_\bot~\text{on}~\dO_0}}\|\bv-\bpsi_h\|.
\end{equation}
On the other hand, due to the finite element inverse inequality it holds
\begin{equation}\label{aux5}
\begin{split}
\|\nabla(\bv-P_h^k(\bv))\|&\le \|\nabla(\bv-\bpsi_h)\| + \|\nabla(\bpsi_h-P_h^k(\bv))\| \\ &
\le \|\nabla(\bv-\bpsi_h)\| + h^{-1}\|\bpsi_h-P_h^k(\bv)\|\\
&\le \|\nabla(\bv-\bpsi_h)\| + h^{-1}(\|\bpsi_h-\bv\|+\|\bv-P_h^k(\bv)\|)
\end{split}
\end{equation}
for arbitrary $\bpsi_h\in\XX_h^k$. Now one applies \eqref{etas} first in \eqref{aux4} and next in \eqref{aux5} to show the following estimate for a smooth $\bv$  satisfying $\Div(J\bF^{-1}\bold{v})= 0$
\begin{equation}\label{approxPh}
\|\bv-P_h^k(\bv)\|+h\|\nabla(\bv-P_h^k(\bv))\| \le c\,h^{m+2}\|\bv\|_{H^{m+\frac52}(\Omega_0)}.
\end{equation}
The orthogonality property \eqref{Orth} and approximation property \eqref{approxPh} for the projection $P_h^k$ are crucial in the proof of an error estimate in the next section.

Before proceeding with the error estimate, we address the question postponed in section~\ref{s_stab}. We need to show that there exists a decomposition $\bu_h^k=\bv_h^k+\bv_{h,1}^k$ satisfying
\eqref{decomp} and \eqref{stab_apr}. To this end, we first note that due to the result in \eqref{bc_est0b} one shows
the bound similar to \eqref{approxPh} as
 \begin{equation}\label{approxPh1}
\|\bv-P_h^k(\bv)\|+h\|\nabla(\bv-P_h^k(\bv))\| \le c\,h^{m+2}\|\bv\|_{C^{m+2}(\overline{\Omega}_0)}.
\end{equation}
 We set $\bv_{h,1}^k=P_h^k(\bv_1(t_k))$, where $\bv_1$ from \eqref{aux_fun} is the divergence-free function, which satisfies the same boundary condition as the original Navier-Stokes solution. Thus, $\bv_{h,1}^k $ satisfies  \eqref{decomp} by construction. The proof of \eqref{stab_apr} relies on \eqref{Orth}, \eqref{approxPh1}, with $m=1$, and the assumption that the norm $\|\bv_1\|_{C^3(Q)}$ is bounded, see section~\ref{s_ext}. 
Using FE inverse inequality we now estimate
\[
\begin{split}
\|\nabla(P_h^k(\bv_1&(t_k)))\|_{L^\infty}\le \|\nabla(\bv_1(t_k)-I_h(\bv_1(t_k)))\|_{L^\infty}+\|I_h(\bv_1(t_k))-P_h^k(\bv_1(t_k)))\|_{L^\infty}\\&\qquad\qquad\qquad+\|\nabla\bv_1(t_k)\|_{L^\infty}\\
&\le \|I_h(\bv_1(t_k))-P_h^k(\bv_1(t_k)))\|_{L^\infty}+2\|\nabla\bv_1(t_k)\|_{L^\infty}+\|\nabla I_h(\bv_1(t_k))\|_{L^\infty}\\
&\le c h^{-\frac{d}2}\|I_h(\bv_1(t_k))-P_h^k(\bv_1(t_k)))\| +C\\
&\le c h^{-\frac{d}2}(\|I_h(\bv_1(t_k))-\bv_1(t_k))\|+\|\bv_1(t_k)-P_h^k(\bv_1(t_k)))\|) +C\\
& \le c h^{2-\frac{d}2}\|\bv_1(t_k)\|_{C^{3}(\overline{\Omega}_0)}+C\le C.
 \end{split}
\]
This proves the first bound in \eqref{stab_apr}. To show the second bound, we note that \eqref{Orth} implies for
$\bv^k_h\in \XX_h^k\cap H^1_0(\Omega_0)^d$ and $\bv^{k-1}_h\in \XX_h^{k-1}\cap H^1_0(\Omega_0)^d$ the following identity
\begin{multline*}
(J_{k-1}(P_h^k(\bv_1(t_k))-P_h^{k-1}(\bv_1(t^{k-1}))),\bv^k_h)=
\Delta t\left\{(J_{k-1}[\bv_1]^k_t,\bv^k_h)\right. \\
\left. +([J]^{k-1}_t(\bv_1(t^{k-1})-P_h^{k-1}(\bv_1(t^{k-1}))),\bv^k_h)\right. \\
\left. +(J_{k-2}(\bv_1(t^{k-1})-P_h^{k-1}(\bv_1(t^{k-1}))),[\bv_h]^k_t)\right\}.
\end{multline*}
Now the desired bound in \eqref{stab_apr} follows by applying the Cauchy inequality, \eqref{approxPh1}, $\|P_h^{k-1}(\bv_1(t^{k-1})\|\le C$ and using the smoothness of $\bxi$ and $\bv_1$ to bound
\[
\|[J]^{k-1}_t\|\le c \|J\|_{C^2(Q)}\le C\quad \text{and}\quad
\|[\bv_1]^{k}_t\|\le c \|\bv_1\|_{C^2(Q)}\le C.
\]

\subsection{Error estimate} \label{s_err}

At time $t=k\Delta t$ the true Navier-Stokes solution $\{\bu, p\}$ satisfies
\begin{multline}\label{FSI_int2}
\left(J_{{k-1}}\left[\bu\right]_t^{k},\bpsi_h\right) +c(\bxi^k;\widetilde{\bw}^k,\bu^{k},\bpsi_h)+a(\bxi^k;\bu^{k},\bpsi_h)
+\frac{1}2(\Div\big(J_{k}\bF^{-1}_{k}\widetilde{\bw}^k\big)\bu^{k},\bpsi_h) \\
+\left(\frac{1}2\left[J\right]^{k}_t\bu^{k},\bpsi_h\right)-b(\bxi^k;p^{k},\bpsi_h)+b(\bxi^k;q_h,\bu^{k})
-(J_{k}{\blf}^{k},\bpsi_h)=\mbox{Approx}(\bpsi_h)
\end{multline}
for all $\bpsi_h\in\VV_h^0,$ $q_h\in\QQ_h$,
with ALE advection velocity {$\widetilde{\bw}^k:=\bu^{k-1}-\left[\bxi\right]^{k}_t$}, and approximation error term
\begin{multline}
\mbox{Approx}(\bpsi_h)=\left(J_{k-1}[\bu]^k_t-J_k\left({\bu}_{t}\right)^k,\bpsi_h\right)+c(\bxi^k;\bw^k,\bu^{k},\bpsi_h) -c(\bxi^k;\widetilde{\bw}^k,\bu^{k},\bpsi_h)\\ +\left([J]_t^{k}\bu^k - \left(J_{t}\right)^k\bu_k,\bpsi_h\right)+\Big(\mbox{div}\left(J_{k}\bF_k^{-1}\widetilde{\bw}^k\right)\bu^k -\mbox{div}\left(J_{k} \bF_k^{-1}\bw^k\right)\bu^k,\bpsi_h\Big).
\end{multline}
We estimate this error in the following lemma.

{
\begin{lemma} Assume the solution $\bu$ of \eqref{weak} is sufficiently smooth such that $\bu_{tt}\in L^{\infty}(Q)^d$, and $\nabla\bu\in L^{\infty}(Q)^{d\times d}$.  We have
\begin{equation}\label{ApproxEst}
|\mbox{\rm Approx}(\bpsi_h)|\le C\,\Delta t\,\|\bD_k(\bpsi_h)\|_{k},
\end{equation}
with a finite constant $C$ depending on $\bxi$ and $\bu$.
\end{lemma}
\begin{proof}
One compares \eqref{FSI_int2} against \eqref{weak} satisfied at time instance $k\Delta t$. The approximation error  contains terms with time-derivatives.
We handle these terms in a standard way with the help of the Taylor expansion in time.
We get
\begin{align*}
 &\left(J_{k-1}[\bu]^k_t-J_k\left({\bu}_{t}\right)^k,\bpsi_h\right) = \left(J_k[\bu]^k_t  + (J_{k-1}-J_k)[\bu]^k_t -J_k\left({\bu}_{t}\right)^k,\bpsi_h\right) =\\
 &= \left(-\frac{1}{\Delta t}J_k\int_{t_{k-1}}^{t_k}{\bu}_{tt}(t-t_{k-1}){\rm d}t -\frac{1}{\Delta t}\int_{t_{k-1}}^{t_k}J_{t}{\rm d}t\cdot\int_{t_{k-1}}^{t_k}\bu_t{\rm d}t,\bpsi_h\right)  =\\ 
 &=-\left(J_k\int_{\tilde{t}_1^k}^{t_k}{\bu}_{tt}{\rm d}t+\bu_t(t_*^k)\int_{t_{k-1}}^{t_k}J_{t}{\rm d}t,\bpsi_h\right)\\
 &
 \le C\,\Delta t\,(\|\bu_{tt}\|_{L^{\infty}(Q)}+\|\bu_{t}\|_{L^{\infty}(Q)}\|J_{t}\|_{L^{\infty}(Q)}) \|\bpsi_h\|_{L^1(\Omega_0)}.
\end{align*}
where $\{\tilde{t}_1^k; t_*^k\}\in [t_{k-1}, t_k]$. We note that for $\bxi\in C^2(Q)^d$ it holds $\|J_{t}\|_{L^{\infty}(Q)}<\infty$.
We estimate the $c$-term using the same arguments,
\begin{align*}
 &c(\bxi^k;\bw^k,\bu^{k},\bpsi_h) -c(\bxi^k;\widetilde{\bw}^k,\bu^{k},\bpsi_h)= \left(J_{k}\nabla\bu^k \bF_k^{-1}\left(\bu^k-\bu^{k-1}\right),\bpsi_h\right) \\
 &= \left(J_{k}\nabla\bu^k \bF_k^{-1}\int_{t_{k-1}}^{t_k}{\bu}_{t}{\rm d}t,\bpsi_h\right)
 \le C\,\Delta t\,\|\bu_{t}\|_{L^{\infty}(Q)}\|\nabla\bu^k\| \|\bpsi_h\|.
\end{align*}
Furthermore,
\begin{align*}
 &\left([J]_t^{k}\bu^k - \left(J_{t}\right)^k\bu_k,\bpsi_h\right)=\left(\bu^k\int_{\tilde{t}_3^k}^{t_k}{J}_{tt}{\rm d}t,\bpsi_h\right)
 \le C\,\Delta t\,\|J_{tt}\|_{L^{\infty}(Q)}\|\bu^k\| \|\bpsi_h\|.
\end{align*}
We note that for $\bxi\in C^2(Q)^d$ it holds $\|J_{tt}\|_{L^{\infty}(Q)}<\infty$.
Finally,
\[
\begin{split}
 \Big(\mbox{div}\left(J_{k}\bF_k^{-1}\widetilde{\bw}^k\right)\bu^k -\mbox{div}\left(J_{k} \bF_k^{-1}\bw^k\right)\bu^k,\bpsi_h\Big)
&
 = \left(\mbox{div}\left(J_{k} \bF_k^{-1}\int_{t_{k-1}}^{t_k}{\bu}_{t}{\rm d}t\right)\bu^k,\bpsi_h\right)\\
 &= -\left(J_{k} \bF_k^{-1}\int_{t_{k-1}}^{t_k}{\bu}_{t}{\rm d}t, \nabla(\bu^k\cdot\bpsi_h)\right)
 \\
 &\le C\,\Delta t\,\|\bu_{t}\|_{L^{\infty}(Q)}\|\bu^k\|_{H^1(\Omega_0)} \|\bpsi_h\|_{H^1(\Omega_0)}.
 \end{split}
\]
Collecting all the terms and applying Korn's inequality \eqref{Korns} to the $\bpsi_h$-terms we prove the lemma.
\end{proof}
}
\medskip

The finite element formulation and \eqref{FSI_int2} imply the equations for the error:
 \begin{multline}\label{FSI_err}
\left(J_{k-1}\left[\be\right]_t^{k},\bpsi_h\right)+c(\bxi^k;\bw^k_h,\be^{k},\bpsi_h) +a(\bxi^k;\be^{k},\bpsi_h)\\
+\frac{1}2(\Div\big(J_{k}\bF^{-1}_{k}\bw^k_h\big)\be^{k},\bpsi_h) +\left(\frac{1}2\left[J\right]^{k}_t\be^{k},\bpsi_h\right) =\mbox{Approx}(\bpsi_h)+\mbox{Consist}(\bpsi_h)
\end{multline}
for all $\bpsi_h\in\XX_h\cap H^1_0(\Omega)^d$,
where the consistency term in the right-hand side reads
\begin{equation} \label{Cons}
  \mbox{Consist}(\bpsi_h)=  \quad b(\bxi^k;p^{k},\bpsi_h)+c(\bxi^k;\bw^k_h,\bu^{k},\bpsi_h)-c(\bxi^k;\bu^k-\bxi_t^k,\bu^{k},\bpsi_h).
\end{equation}
Decompose the error
\[\be^k= \be_{I}^k+\be_{h}^k= \{\bu^k-P_h^k(\bu^k)\}+\{P_h^k(\bu^k)-\bu^k_h\}.
\]
We set $J_{-1}=1$ to define $P_h^0$.
Thanks to the definition of $P_h^k$ and boundary conditions for $\bu_h^k$, the vector field $\be_{h}^k$ vanishes on $\dO_0$.
Thus,  we can set $\bpsi_h=\be_{h}^{k}$  in \eqref{FSI_err} and apply the same arguments that were used to show~\eqref{balance}. We obtain the estimate
\begin{multline}\label{error1}
\frac{1}2\|\be_{h}^{k}\|_{k}^2 + \frac{(\Delta t)^2}2\|[\be_h]^{k}_t\|_{k}^2
+2\nu \Delta t \|\bD_{k}(\be_{h}^{k})\|_{k}^2 
\le
\frac{1}2\|\be_{h}^{k-1}\|_{k-1}^2\\+ \Delta t\left\{\mbox{Approx}(\be_{h}^{k})+\mbox{Consist}(\be_{h}^{k})+\mbox{Interp}(\be_{h}^{k})\right\}.
\end{multline}
The last term accumulates the integrals with the finite element interpolation error
\begin{multline*}
\mbox{Interp}(\be_{h}^{k})= c(\bxi^k;\bw^k_h,\be_{I}^{k},\be_{h}^{k})
+a(\bxi^k;\be_{I}^{k},\be_{h}^{k}) \\ +\left(\frac{1}2\left[J\right]^{k}_t\be_{I}^{k},\be_{h}^{k}\right)
+\left(\frac{1}2J_{k-2}\be_{I}^{k},[\be_h]^{k}_t\right)+\frac{1}2(\Div\big(J_{k}\bF^{-1}_{k}\bw^k\big)\be_{I}^{k},\be_{h}^{k}).
\end{multline*}
Here we exploited the properties of the projections $P_h^{k-1}$, $P_h^k$:
\[
\left(J_{k-2}\be_{I}^{k-1},\be_{h}^{k-1}\right)=0, \quad\left(J_{k-1}\be_{I}^{k},\be_{h}^{k}\right)=0
\]
respectively, and so
$\left(J_{k-1}\left[\be_{I}\right]_t^{k},\be_{h}^{k}\right)=-\left(\left[J\right]^{k}_t\be_{I}^{k},\be_{h}^{k}\right)-
\left(J_{k-2}\be_{I}^{k},[\be_h]^{k}_t\right)$.

Now we estimate the consistency and interpolation terms on the right-hand side of   \eqref{error1}.
First we treat the term $\mbox{Consist}(\be_{h}^k)$.
Denote by $I(p^{k})\in\QQ_h$ a suitable interpolant for $p^{k}$.   Due to $\be_h^k\in \XX_h^k$, the polynomial interpolation properties and Korn's inequality, we have
\begin{equation} \label{Cons1}
\begin{split}
b(\bxi^k;p^{k},\be_{h}^{k})&=b(\bxi^k;p^{k}-I(p^{k}),\be_{h}^{k})\le c\,\|p^{k}-I(p^{k})\|\|\nabla\be_{h}^{k}\|\le
C\,h^{m+1}\|p^{k}\|_{H^{m+1}(\Omega)}\|\bD_{k}(\be_{h}^{k})\|_k\\
&\le C\left( \delta^{-1} h^{2(m+1)}\|p^{k}\|^2_{H^{m+1}(\Omega)}+ \delta \|\bD_{k}(\be_{h}^{k})\|_{k}^2\right)
\end{split}
\end{equation}
for any $\delta>0$.

We apply H\"{o}lder's inequality and Sobolev embedding inequalities as well as interpolation properties of
the finite element projection in \eqref{approxPh} to handle other consistency and interpolation terms.
Using $\|\nabla\bu\|_{L^{\infty}(\Omega)}\le C$ and \eqref{assumption}, we get
\begin{equation} \label{Cons2}
\begin{split}
c(\bxi^k;\bw^k_h,\bu^{k},\be_{h}^{k})-c(\bxi^k;{\widetilde{\bw}^k},\bu^{k},\be_{h}^{k})&=c(\bxi^k;\be^{k-1},\bu^{k},\be_{h}^{k})\le C\|\be^{k-1}\|_k\|\be_{h}^{k}\|_k \\ & \le C(\|\be^{k-1}_h\|_k^2+\|\be^{k-1}_I\|_k^2+\|\be_{h}^{k}\|_k^2)\\
&\le C(\|\be^{k-1}_h\|_k^2+\|\be_{h}^{k}\|_k^2+h^{2(m+1)}).
\end{split}
\end{equation}
Now we treat the terms  contributing to the  interpolation error. Using also  stability estimate \eqref{Stab_u} with $c\Delta t\ge h^{2m+4}$, we bound
\begin{equation} \label{extra1}
c(\bxi^k;\bw^k_h,\be_{I}^{k},\be_{h}^{k})\le C\,(\|\bv^k_h\|_k+1)\|\be_{I}^{k}\|\|\be_{h}^{k}\|_k\le C(\|\be_{h}^{k}\|_k^2+h^{2(m+1)}),
\end{equation}
\begin{equation} \label{extra2}
a(\bxi^k;\be_{I}^{k},\be_{h}^{k}) \le C\,h^{m+1}\|\bD_{k}(\be_{h}^{k})\|_{k} \le  C\left(\delta^{-1} h^{2(m+1)}+ \delta \|\bD_{k}(\be_{h}^{k})\|_{k}^2\right).
\end{equation}
Using $\bxi\in C^2(Q)$ and \eqref{approxPh}, we get
\begin{equation} \label{extra3}
\begin{split}
\left(\left[J\right]^{k}_t\be_{I}^{k},\be_{h}^{k}\right)+\left(J_{k-2}\be_{I}^{k},[\be_h]^{k}_t\right)&\le C\,(\|J_t'\|_{L^\infty(Q)}+\|J\|_{L^\infty(Q)})\|\be_{I}^{k}\|(\|\be_{h}^{k}\|_k+\|[\be_h]^{k}_t\|) \\ &\le C(\|\be_{h}^{k}\|_k^2+(\Delta t)^{-1}h^{2(m+2)})+ \frac{\Delta t}{2}\|[\be_h]^{k}_t\|^2.
\end{split}
\end{equation}
Summarizing \eqref{error1}--\eqref{extra3} and using \eqref{ApproxEst}, we get for $\delta$ small enough:
\begin{multline} \label{extra4}
\frac{1}2\|\be_{h}^{k}\|_{k}^2
+2\nu \Delta t \|\bD_{k}(\be_{h}^{k})\|_{k}^2
\\ \le
\frac{1}2\|\be_{h}^{k-1}\|_{k-1}^2+ \tilde{C}\Delta t\left\{\|\be_{h}^{k}\|^2+\|\be_{h}^{k-1}\|^2+h^{2(m+1)}+(\Delta t)^2+(\Delta t)^{-1}h^{2(m+2)}\right\},
\end{multline}
where $\tilde{C}$ depends only on the data.
Now we assume the following:
\begin{equation} \label{assump3}
\text{either}~~\Delta t~~ \text{is small enough s.t.}~~\small{\frac12}-\tilde{C}\Delta t>0~~~\text{or}~~~ \nu\ge \tilde{C}\,C_K,
\end{equation}
 where $C_K$ is Korn's inequality constant from \eqref{Korns}.  Note that $\be_{h}^{0}=\bu^0_h-P_h^k(\bu(0))$, where $\bu_0^h$ is the Lagrange interpolant of $\bu(0)$,
and so $\|\be_{h}^{0}\|_0^2\le Ch^{2(m+2)}$. We sum \eqref{extra4} over all $k=1,\dots,n$, $n\le N$, and apply the discrete Gronwall inequality  to get
\[
\|\be_{h}^{n}\|_{k}^2+ 2\nu \Delta t \sum_{k=1}^n\|\bD_{k}(\be_{h}^{k})\|_{k}^2\le C\left(h^{2(m+1)}+(\Delta t)^2+
(\Delta t)^{-1}h^{2(m+2)}\right)
\]
with some $C$ independent of $n$, $h$, $\Delta t$.
Applying the triangle inequality and \eqref{approxPh} one more time, we obtain the final error bound in the energy norm,
\begin{equation} \label{err_est}
\|\be\|^2_\ast
:=\max_{1\le k\le N}\|\be^{k}\|_{k}^2+ 2\nu \Delta t \sum_{k=1}^N\|\bD_{k}(\be^{k})\|_{k}^2\le C\left(h^{2(m+1)}+(\Delta t)^2+(\Delta t)^{-1}h^{2(m+2)}\right).
\end{equation}
 The main result is summarized in the following theorem.

\begin{theorem}\label{Th1} For the fluid problem \eqref{NSE}--\eqref{cont_int} and the finite element method \eqref{FE1}
assume the following:
\begin{enumerate}
\item the domain evolution is given by a smooth mapping which satisfies \eqref{assumption} and \eqref{assum2};
\item $\dO(t)=\dO^{ns}(t)$ for all $t\in[0,T]$;
\item $\bxi_t\circ\bxi^{-1}$ on $\bigcup\limits_{t\in[0,T]}\dO(t)\times\{t\}$ is the trace of  some $\hat{\bv}_1\in C^3(Q^{\rm phys})$ s.t. $\Div\hat{\bv}_1=0$;
\item  $\Omega_0$ is a convex polyhedron;
\item $\bu_{tt}\in L^\infty(\Omega_0)$, $\bu(t)\in H^{m+\frac52}(\Omega_0)$,  $p(t)\in H^{m+1}(\Omega_0)$ for all $t\in[0,T]$;
\item $c\Delta t\ge h^{2m+4}$ with some $c$ independent of $h$, $\Delta t$;
\item condition \eqref{assump3} holds.
\end{enumerate}
Then the finite element method is stable, see \eqref{Stab_u},  and the error estimate  \eqref{err_est} holds.
\end{theorem}

For the popular P2-P1 Taylor--Hood element we have $m=1$, and  the optimal error bound $\|\be\|_\ast \le C\max\{h^{2}; \Delta t\}$ follows if $h^2\le c \Delta t$.
The second order in time approximation can be achieved in practice (but not analyzed here) by using BDF2 time stepping instead of backward Euler time stepping.

\rev{We note that constant $C$ in \eqref{err_est} depends on problem data and may blow up for large Reynolds numbers. As it is common for FE methods for fluids, the present method requires stabilization or subgrid scales modelling for higher Re numbers. In section~\ref{s_num2} we consider the extension of our quasi-Lagrangian formulation to  Smagorinski eddy viscosity model.}

\section{Numerical experiments}\label{s_num}

In this section we present numerical results for the finite element method \eqref{FE1} implemented within the open source Ani3D software (www.sf.net/p/ani3d) applied to two flow examples.
In the first experiment we demonstrate experimental convergence rates for a given analytical solution and compare those to the established theoretical expectations.
The second experiment concerns with a blood flow in a geometrical dynamic model of the human left ventricle.
The motion of the ventricle was extracted from a sequence of ceCT images of a real patient heart over one cardiac cycle.

\subsection{Convergence to analytical solution}

In cylindrical coordinates $(r, y, \phi)$, we define the reference (initial) domain $\Omega_0$ to be the axisymmetric tube with symmetry axis $0y$:
$$
\Omega_0 =\{ (r, y, \phi): -4\le y \le 4, r^2\le e^{y/4+1} \}.
$$
The domain has outflow boundary $\partial \Omega^N_0 = \partial\Omega_0 \cap \{(r, y, \phi): y=4\}$ and
no-slip no-penetration boundary $\partial \Omega_0^{ns} = \partial\Omega_0 \setminus \partial \Omega^N_0$.

The physical time-dependent domain is:
$$
\Omega(t) =\{ (r, y, \phi): -4\le y \le 4, r^2\le e^{y/4+1}(1-\frac14 t) \},\quad  t\in[0,0.2].
$$

The analytical solution $\{\bu,p\}$ to \eqref{NSE}-\eqref{bc_DN} is given in cylindrical coordinates $(r, y, \phi)$ by
\begin{align*}
&u_r = -\frac{2e^{-\frac{1}{4}(y+4)}r^{3}}{(4-t)^2},\quad u_y = \frac{8}{4-t} - \frac{32 e^{-\frac{1}{4}(y+4)}r^2}{(4-t)^2},\quad u_{\phi} = 0,\\
&p = 512\nu\frac{e^{-\frac{1}{4}(y+4)}}{(4-t)^2}-8\frac{y}{(t-4)^2} + \tilde{p}(t),
\end{align*}
where $\tilde{p}(t)$ depends only on $t$. To ensure unique pressure, we set $p=0$ on $\partial \Omega^N(t)$.

The external force $\blf = (f_r,f_y,f_\phi)$ is taken such that $\bu$  satisfies  \eqref{NSE}-\eqref{constit_f} \rev{for $\nu=0.04$}..
To simplify the design of the right-hand side and the boundary conditions,
we consider in \eqref{constit_f} the full velocity gradient tensor $\nabla\bu\bF^{-1}$ instead of its doubled symmetric part. The resulting forcing vector is given by
\[
f_r = \nu\frac{e^{-\frac{1}{4}(y+4)}}{(4-t)^2}\left(16r+\frac{1}{8}r^3\right)-4\frac{e^{-\frac{1}{2}(y+4)}}{(t-4)^4}r^5,\quad
f_y = 2\nu\frac{e^{-\frac{1}{4}(y+4)}}{(4-t)^2}r^2-128\frac{e^{-\frac{1}{2}(y+4)}}{(t-4)^4}r^4,\quad f_{\phi} = 0.
\]
The  boundary condition  on $\partial \Omega^N(t)$  is non-homogeneous, $\bsigma \bn = \nu J\nabla\bu\bF^{-1}\bF^{-T}\bn_N$, $\bn_N:=(0,1,0)$.

We applied the lower degree ($m=1$) $P_2$-$P_1$ Taylor-Hood finite elements in \eqref{FE1}
with the simplified stress tensor \eqref{constit_f}.
A sequence of six unstructured quasiuniform tetrahedral meshes with  mesh sizes $h_i = h_{i-1}/\sqrt{2}$, $i=2,\dots,6$ is used
to measure the finite element error.
The computed error norms are shown in Table \ref{table:hyper_complex}.
The second order asymptotic convergence rate are consistent with the estimate of Theorem~\ref{Th1}.

\begin{table}[!ht]
  \begin{center}
    \begin{tabular}{*{7}{c}}
    \hline
Mesh step size    & 1.0 & $1.0/\sqrt{2}$ & 0.5  & $0.5/\sqrt{2}$ & 0.25 & $0.25/\sqrt{2}$ \\
    \hline
     Number of mesh elements & 389 & 928 & 2333 & 5813 & 16439 & 46215 \\
    \hline
Time step size    & 0.04 & 0.02 & 0.01 & 0.005 & 0.0025 & 0.00125\\
    \hline
    Number of time steps $N$ & 5 & 10 & 20 & 40 & 80 & 160\\
    \hline
$\underset{1\leqslant k\leqslant N}{\mbox{max}}\|\be^{k}\|_{k}+ \sqrt{\sum_{k=1}^N\Delta t \|\bD_{k}(\be^{k})\|_{k}^2}$ & 0.2652 & 0.1731 & 0.0983 & 0.0534 & 0.0233 & 0.0115 \\
    \hline
 Error ratio  & & 1.532 & 1.761 & 1.841 & 2.292 & 2.034\\
    \end{tabular}
  \end{center}
\caption{Finite element error for the given analytic solution. \label{table:hyper_complex}}
\end{table}

\subsection{Blood flow in a geometrical model of the human left ventricle}\label{s_num2}

We now illustrate the practical value of  the finite element  method \eqref{FE1} 
by applying it to simulation of the blood flow in a simplified model of the human left ventricle.
This problem arises in patient-oriented hemodynamic applications.
Our simplifications concern  anatomical structure and boundary conditions on the valves as well as neglecting ventricle twisting.

The motion of the ventricle $\Omega(t)$ is recovered from a sequence of ceCT images.
The input data is a dataset of 100 images with $512 \times 512 \times 480$ voxels and  $0.625 \times 0.625 \times 0.25$~mm resolution. The images were taken from a chest ceCT of a 50 years old female.
The mapping $\bxi$ is defined by a dynamic sequence of 1981 meshes for one cardiac cycle. The mesh sequence contains
 \textit{topologically invariant} tetrahedral meshes with 14033 nodes, 88150 edges and 69257 tetrahedra which differ  only in nodes positions.
Denote by $\bx$  the spatial coordinate of a node of the reference grid at time $t=0$, and identify the coordinate $\bxi (\mathbf{x}, t)$ of the corresponding node  at time $t=t_k$.
The mapping $\bxi^k$ is defined as the continuous piecewise linear vector function with values $\bxi (\mathbf{x}, t)$ at the reference grid nodes $\bx$. Of course, the recovered  mapping is a piecewise smooth approximation of an unknown smooth mapping $\bxi$. This introduces additional modelling error which is not treated in our analysis.

Note that the time between two sequential frames of ceCT input data is equal to {12.7 milliseconds}. Because of the fast heart walls motion and variation of ventricle volume (see Fig.~\ref{figure:volume})
setting $\Delta t=12.7$\,ms turned out to be too large to deliver acceptable accuracy.
The sequence of 1981 meshes  allows us to use 20 times smaller time step $ \Delta t = 0.635$\,ms.
For the details of generation of this sequence we refer to recent paper \cite{DLOVpaper}.

\begin{figure}[!htbp]
    \includegraphics[width=\textwidth, clip=True]{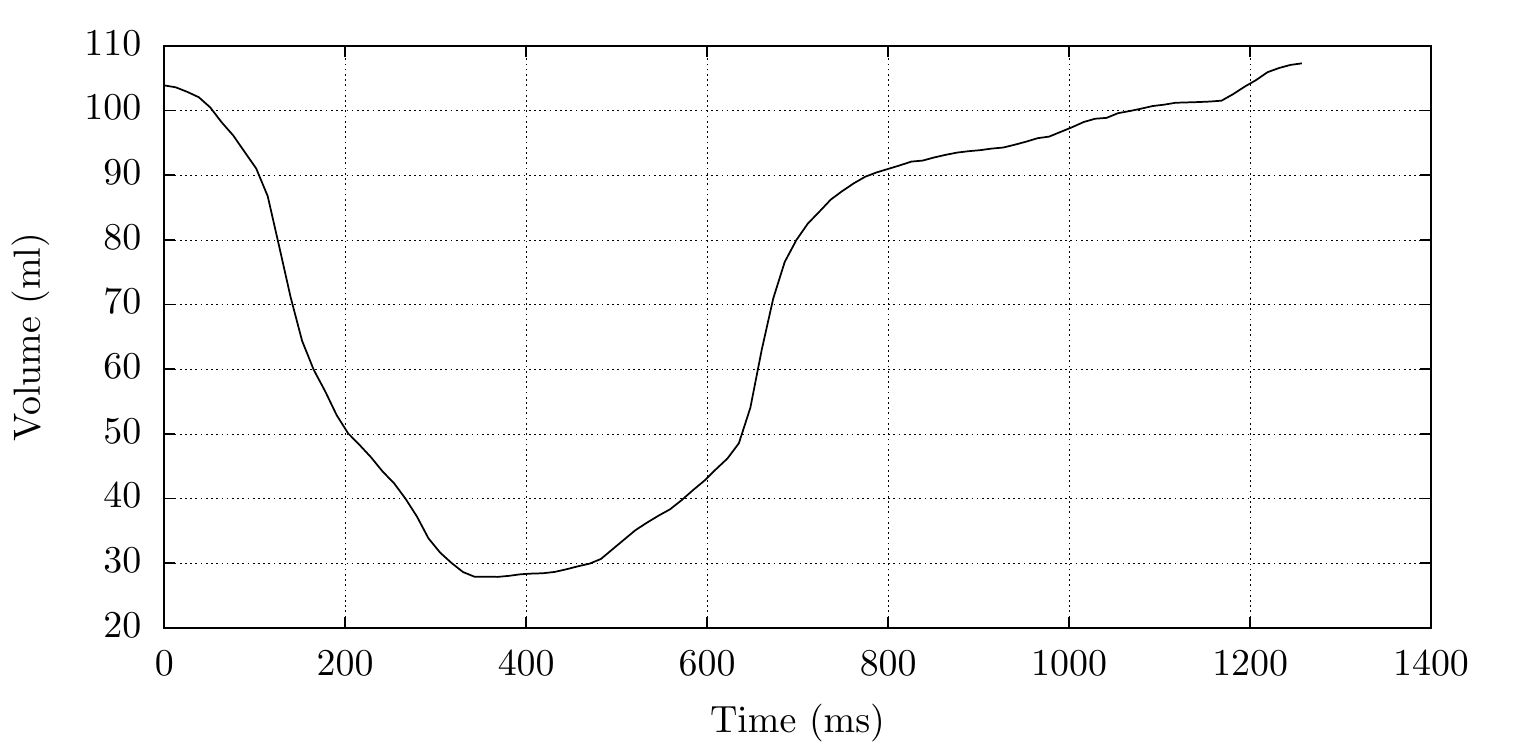}
    \caption{The ventricle volume change in time.}
    \label{figure:volume}
\end{figure}

To set up the boundary conditions, we split
the left ventricle boundary into  aortic valve and mitral valve patches, and the remaining part of the boundary.
The ventricle passes through the systole phase approximately until {$t = 355 $ ms} releasing blood flow through the aortic valve.
During this time interval, we set the `do-nothing' boundary condition \eqref{bc_DN} with $\hat\bg = \mathbf{0}$  on the patch associated with the aortic valve.
For the remaining time interval ending at $ T = 1.2573 $ s, we impose the `do-nothing' boundary condition on the patch associated with the mitral valve.
The latter is connected to the atrium and intakes  blood during the expansion stage called diastole.
On the remaining boundary, including  the aortic valve during the diastole phase and the mitral valve  during the systole phase, the no-penetration no-slip condition
\eqref{cont_int} is imposed, $\bu = \bxi_t$.
In future, we plan to use more physiologically suitable boundary conditions, see \cite{tagliabue2015}, instead of the `do-nothing' boundary conditions.

At the initial time moment, the beginning of the systole stage,  we assume the system is at rest with zero pressure (this is, of course, an idealized initial condition).
Similarly to the previous test case, we use the lower degree Taylor--Hood finite element spaces \eqref{defVQ} with $m=1$.
At every time step one has to solve a linear system with 320582 unknowns,
comprised of 14033 nodal pressure degrees of freedom and 14033$+$88150  degrees of freedom residing at the mesh vertices and edge centers, for each velocity component.

The hemodynamics of the  heart is characterized by transitional or even turbulent  flow regimes (see, e.g., \cite{querzoli2010effect,falahatpisheh2012high,chnafa2014image}).
Our  mesh is not sufficiently fine to resolve all scales in the flow, and hence a subgrid model has to be employed.
For the purposes of this paper, we use the simplest  approach and replace $\nu$ with the  Smagorinsky turbulent  viscosity coefficient $\nu_\tau$:
$$
\nu_\tau = \nu + (C_s h_T)^2 \sqrt{2\bD_k(\bw^k_h) : \bD_k(\bw^k_h)},
$$
where $C_s=0.2$ and $h_T = \textrm{diam} (T)$, for any tetrahedral cell $T\in \mathcal{T}_h$.
The finite element method takes the form: Find velocity $\bu^k\in \VV_h$ and pressure $p^k\in \QQ_h$ satisfying equation
\begin{align}\label{E:6}
&\int_{\Omega_0}J_{k}\frac{\mathbf{u}^{k}-\mathbf{u}^{k-1}}{\Delta t}\cdot\bpsi\,\mathrm{d}\bx  + \int_{\Omega_0}J_{k}\nabla\mathbf{u}^{k}\mathbf{F}_{k}^{-1}\left(\mathbf{u}^{k-1}-\frac{\bxi^{k}-\bxi^{k-1}}{\Delta t}\right)\cdot \bpsi\,\mathrm{d}\bx
  -\int_{\Omega_0}J_{k}p^{k}\mathbf{ F}_{k}^{-T}:\nabla\bpsi\,\mathrm{d}\bx  \notag \\
& + \int_{\Omega_0}J_{k}q\mathbf{ F}_{k}^{-T}:\nabla\mathbf{u}^k\,\mathrm{d}\bx
  + \int_{\Omega_0}\nu_\tau J_{k}(\nabla\mathbf{u}^{k}\mathbf{ F}_{k}^{-1}\mathbf{ F}_{k}^{-T}+\mathbf{ F}_{k}^{-T}(\nabla\mathbf{u}^{k})^{T}\mathbf{ F}_{k}^{-T}):\nabla\bpsi\,\mathrm{d}\bx = 0
\end{align}
and the no-penetration no-slip  $\bu^k = (\bxi^{k}-\bxi^{k-1})/\Delta t$ or the `do-nothing' boundary conditions, for all $\bpsi$ and $q$ from the corresponding FE spaces.

The cutaway of the  mesh and computed velocity streamlines and the Q-criterion field
at 200, 400, 600\,ms are shown in Figs.~\ref{figure:t200}, \ref{figure:t400}, \ref{figure:t600}, respectively. \rev{The velocity magnitude is given in $mm/sec$. The viscosity of blood was set $\nu=4mm^2/sec$.}
These instances are chosen to demonstrate blood flow features in the middle of the systole phase,
the very beginning of the diastole phase, and the beginning of the second quarter of the diastole phase (Fig.~\ref{figure:volume}).
The velocity streamlines and the Q-criterion are shown in the entire 3D domain, whereas the mesh is shown partly: only cells lying beyond the cross-section plane are visible.

\begin{figure}[!htbp]
    \includegraphics[width=\textwidth, clip=True]{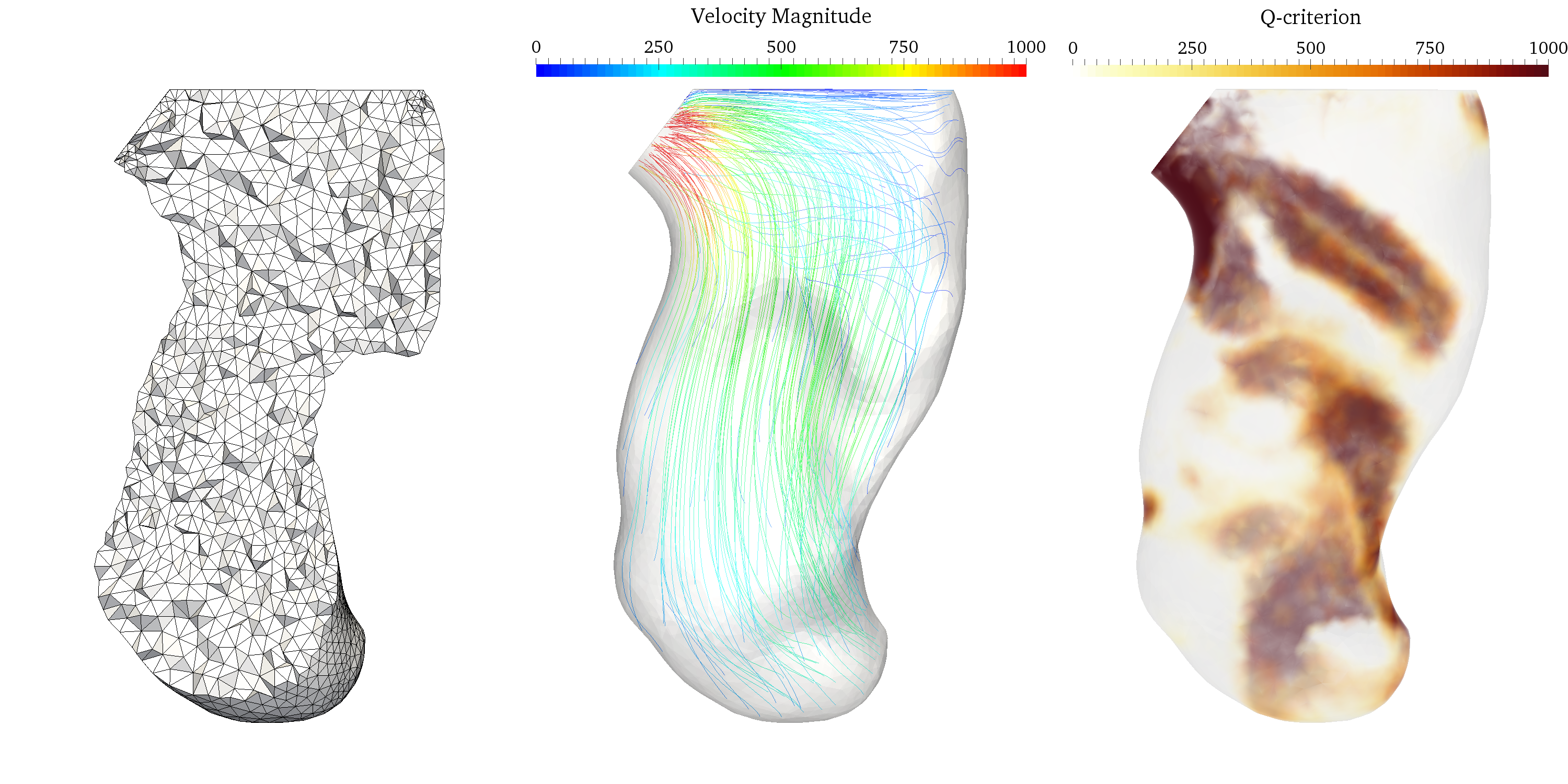}
    \caption{The cutaway of the  mesh, velocity streamlines and the Q-criterion field at  $t = 200$ ms, horizontal long axis view.}
    \label{figure:t200}
\end{figure}

\begin{figure}[!htbp]
    \includegraphics[width=\textwidth, clip=True]{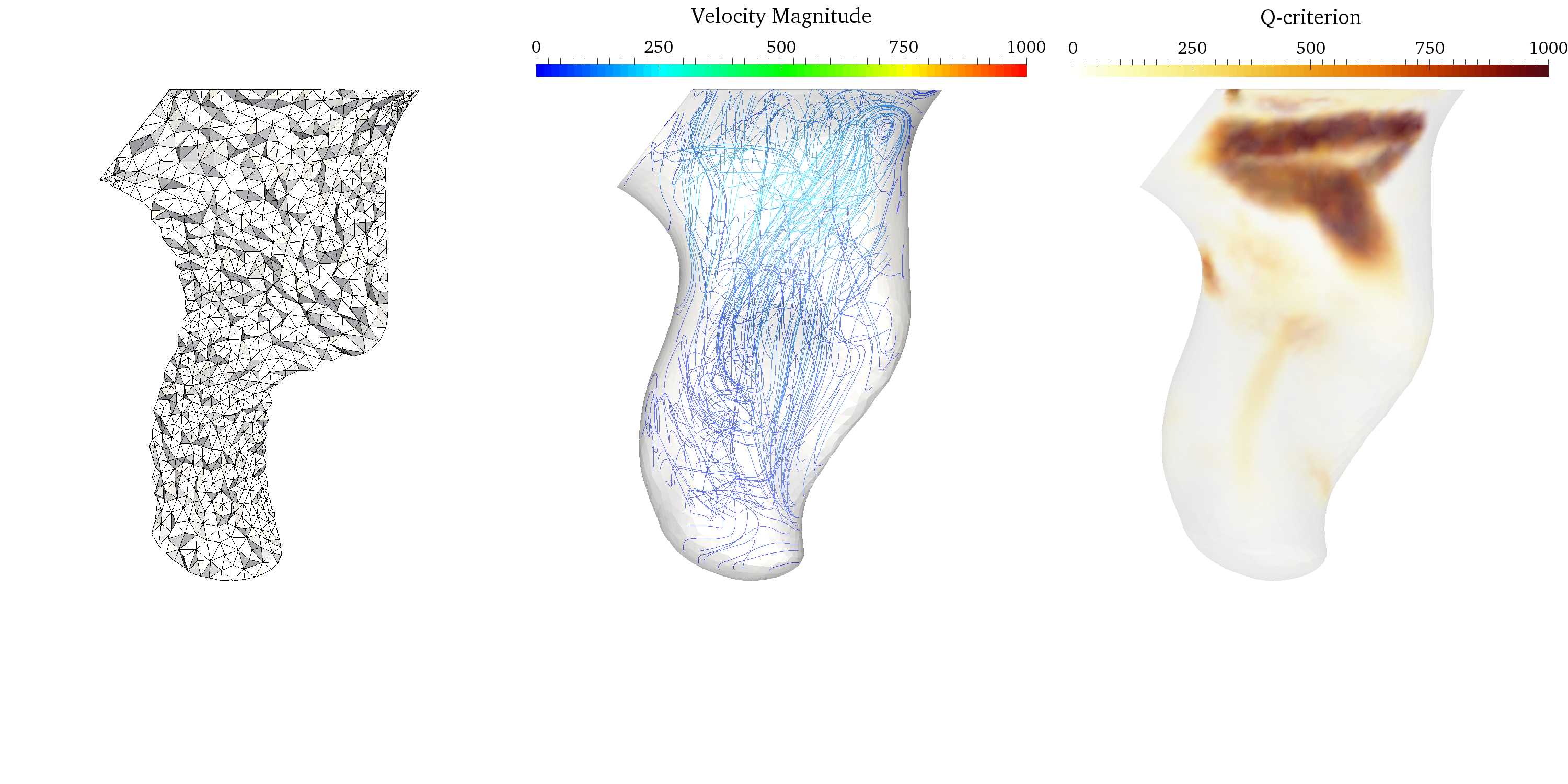}
    \caption{The cutaway of the  mesh, velocity streamlines and the Q-criterion field at  $t = 400$ ms, horizontal long axis view.}
    \label{figure:t400}
\end{figure}

\begin{figure}[!htbp]
    \includegraphics[width=\textwidth, clip=True]{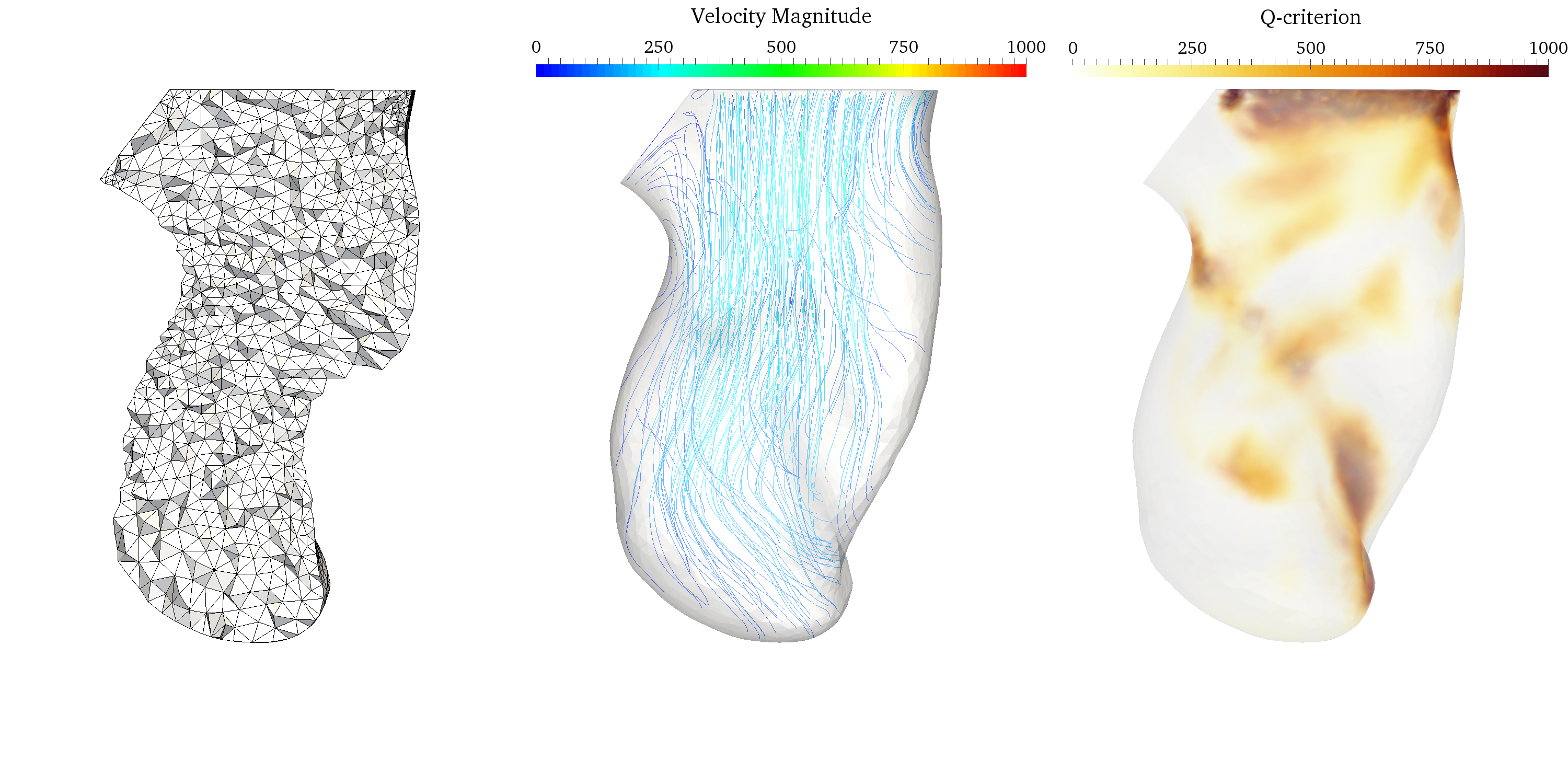}
    \caption{The cutaway of the  mesh, velocity streamlines and the Q-criterion field at  $t = 600$ ms, horizontal long axis view.}
    \label{figure:t600}
\end{figure}

\section{Conclusions} \label{s_outlook}  The paper shows that the quasi-Lagrangian description of the flow in a time-dependent domain is suitable for a practically efficient finite element method, which is amenable to mathematically rigorous numerical analysis. Major challenges in analysis are dealing with the time-dependent divergence-free constraint in finite element spaces and handling non-homogeneous boundary conditions. This paper proves the optimal error estimate in the energy norm under fairly practical assumptions. In a relatively straightforward way, implementation of the method   builds on standard Navier-Stokes finite element solvers in a fixed triangulated domain. The example of the flow in the simplified model of the human left ventricle demonstrated the practical value of the method for certain medical applications.

\subsection*{Acknowledgements}  We would like to thank Dr. Alexander Danilov for providing the sequence of meshes and postprocessing the solution for the experiment with the flow in the simplified model of the human left ventricle.

\bibliographystyle{elsarticle-num}
\bibliography{mybib}

\end{document}